\documentclass[10pt,reqno]{elsarticle}
\usepackage{amsmath,amssymb,amsthm}
\usepackage[latin1]{inputenc}
\usepackage{version,tabularx,multicol}
\usepackage{graphicx,float, url}
\usepackage{amsfonts}
\usepackage{soul}
\usepackage{pgf}
\usepackage{tikz}
\usetikzlibrary{arrows}
\usetikzlibrary{shapes}
\usepackage{psfrag}
\usepackage{bm}
\usepackage{MnSymbol}

\usepackage{graphicx}
\usepackage{epsfig}
\usepackage{subfigure}
\usepackage{float}
\usepackage{flafter}

\usepackage{ifpdf}
\usepackage{color}
\definecolor{webgreen}{rgb}{0,.5,0}
\definecolor{webbrown}{rgb}{.8,0,0}
\definecolor{emphcolor}{rgb}{0.95,0.95,0.95}

\usepackage{url}

\theoremstyle{plain}
\newtheorem{theorem}{Theorem}[section]
\newtheorem{prop}[theorem]{Proposition}
\newtheorem{lem}[theorem]{Lemma}

\newtheorem{cor}[theorem]{Corollary}

\newtheorem{rem}[theorem]{Remark}
\newtheorem{defn}[theorem]{Definition}

\theoremstyle{definition}
\newtheorem{example}[theorem]{Example}

\theoremstyle{remark}
\newtheorem{remark}[theorem]{Remark}

\newcommand{\N}{\ensuremath{\mathbb{N}}}

\newcommand{\R}{\ensuremath{\mathbb{R}}}
\renewcommand{\P}{\ensuremath{\mathbb{P}}}
\newcommand{\E}{\ensuremath{\mathbb{E}}}

\newcommand{\be}{\begin{equation}}
\newcommand{\ee}{\end{equation}}

\vspace{0.3cm}

\begin{document}

\begin{frontmatter}

\title{The ancestral selection graph for a $\Lambda$-asymmetric Moran model}

\author[AG]{Adri\'an Gonz\'alez Casanova}
\ead{adrian.gonzalez@im.unam.mx,gonzalez.casanova@berkeley.edu}

\author[NK]{Noemi Kurt}
\ead{kurt@math.uni-frankfurt.de}
%\cortext{Corresponding author}

\author[JLP]{Jos\'e Luis P\'erez}
\ead{jluis.garmendia@cimat.mx}

\address[AG]{Instituto de Matematicas, Universidad Nacional Autonoma de Mexico (UNAM), Cuernavaca, M\'exico and Department of Statistics, University of California at Berkeley}

\address[NK]{Institut f\"ur Mathematik, Johann Wolfgang Goethe-Universit\"at,  60325 Frankfurt am Main,  Germany}

\address[JLP]{Department of Probability and Statistics, Centro de Investigaci\'on en Matem\'aticas A.C. Calle Jalisco s/n. C.P. 36240, Guanajuato, Mexico}

\date{\today}

\begin{abstract}
Motivated by the question of the impact of selective advantage in populations with skewed reproduction mechanims, we study a Moran model with selection. We assume that there are two types of individuals, where the reproductive success of one type is larger than the other. The higher reproductive success may stem from either more frequent reproduction, or from larger numbers of offspring, and is encoded in a measure $\Lambda$ for each of the two types.  $\Lambda$-reproduction here means that a whole fraction of the population is replaced at a reproductive event. Our approach consists of constructing a $\Lambda$-asymmetric Moran model in which individuals of the two populations compete, rather than considering a Moran model  for each population.  Provided the measure are ordered stochastically, we can couple them. This allows us to construct the central object of this paper,  the $\Lambda-$asymmetric ancestral selection graph,  leading to a pathwise duality of the forward in time $\Lambda$-asymmetric Moran model with its ancestral process. We apply the ancestral selection graph in order to obtain scaling limits of the forward and backward processes, and note that the frequency process converges to the solution of an SDE with discontinous paths. Finally, we derive a Griffiths representation for the generator of the SDE and use it to find a semi-explicit formula for the probability of fixation of the less beneficial of the two types.
\end{abstract}


\end{frontmatter}

\section{Introduction}
There is a deep connection between the ancestry of a population and the dynamics of its genetic configuration. Mathematical population genetics exploits and formalises this connection, see for example \cite{Etheridge-Notes} for an overview. Its most classical instance is the moment duality relation between the block counting process of the Kingman coalescent and the Wright Fisher diffusion, which relates the past and future of a population that evolves on absence of selection and in which the number of offsprings of a mother is of a smaller order of magnitude than the population size. 

This connection between the past and the future is ubiquitous and occurs in many biologically inspired mathematical models. For example,  the ancestry of populations with skewed offspring distributions have been modelled by $\Lambda$-coalescents or multiple merger coalescents \cite{Pitman, Sagitov} different from the Kingman coalescent, but which still have a moment dual, namely the $\Lambda$-Fleming Viot process (see \cite{BB,D-K1, D-K2}).  Skewed offspring distributions occur in populations where the number of offspring of one mother can be of the order of magnitude of the population size.  It is believed that they may be relevant for certain marine species such as the Atlantic cod, where one mother may lay millions of eggs. See e.g. \cite{AKHE, EW} and references therein for a more detailed discussion.

In the presence of weak selection, for example when one type of individuals reproduces slightly faster than the others, it is cumbersome to formally connect the ancestry and the forward in time changes in a population configuration.  Remarkably, there is a notion of potential ancestry that overcomes this difficulty for populations in the universality class of the Kingman coalescent: The celebrated ancestral selection graph (ASG) of Krone and Neuhauser \cite{KN, NK}. Heuristically, in a population with only two types, Krone and Neuhauser  were able to describe the number of potential ancestors of a sample as a branching coalescing process. Under the rule that an individual is of selective type if at least one of its ancestors has selective advantage, the forward in time propagation of types can be specified in terms of the potential ancestry and the types at time zero. As a consequence, the frequency of individuals with selective disadvantage is a Markov process, which is in moment duality to the process of the number of potential ancestors (backward in time). The graphical construction of the ancestral selection graph provides a pathwise duality relation of the forward and backward processes, which for example links fixation probabilities with the ancestral process, see e.g. \cite{PP, KB}.

 However, the classical ancestral selection graph couldn't capture genetic dynamics of a population of, say, cod in which  a subpopulation is capable of reproducing faster, as this leads to selective events individually affecting large portions of the population.  What is then a good model for populations with skewed offspring distributions in the presence of selection?  This was one of the main motivating questions of work by Griffiths, Etheridge and Taylor \cite{EGT}, see also Etheridge and Griffiths \cite{EG}.  These authors showed that the process of potential ancestors in this case should be a nonlinear branching and coalescing process. They equipped their model with parent independent mutation and found a duality in terms of the stationary distribution of the forward process, in the case that the stationary distribution exists. They also spelled out the reason for this mysterious nonlinear branching.
 
Furthermore, it was theorised by Gillespie \cite{Gill73,Gill74, Gill75} that subpopulations with different reproduction mechanisms competing in a sequential sampling experiment (for example, in the Lenski experiment \cite{Lenski, GKWY}) lead to different macroscopic behaviours, which scale beyond the Wright-Fisher diffusion universality class. A detailed mathematical analysis of these \textit{asymmetric} models was carried out in \cite{CGP}. What does it mean that one subpopulation has selective advantage over another if they have very different reproduction behaviour? Can one compare the strategy of the cod with the strategy of the rabbit?

 The goal of the present paper is to revisit the motivating questions of \cite{EGT} and \cite{Gill73, Gill74, Gill75} to give a graphical representation of a Moran model with $\Lambda$-reproduction in the presence of selection and asymmetry, following the structure of the ancestral selection graph. Our representation provides a pathwise duality between what we will call below the $\Lambda$-asymmetric Moran model and its ancestral line counting process, from which the generator duality can be derived.  Contrary to Etheridge, Griffiths and Taylor we don't need to start from the invariant distribution in our construction and therefore don't need to include mutations in the model.  Moreover, and probably most importantly,  our construction works without previously assigning types to the individuals,  since we can construct the ancestral selection graph independently, in the spirit of Neuhauser and Krone.

 Our construction is based on a coupling of the measures governing the reproduction mechanisms of two subpopulations,  which also has an interpretation in terms of optimal transport theory. However, not every pair of measures can be used to construct a $\Lambda$-ancestral selection graph, and we provide explicit sufficient conditions. This is done by %introducing the \textit{partial order of adaptation} and
 	 showing that a pair of probability measures that is comparable with respect to the stochastic partial order can be coupled to construct an ancestral selection graph.  This can be generalised to finite measures ordered in a similar manner.

Finally, we derive a representation inspired by Griffiths \cite{Griff} for the processes that arise as scaling limits of  our $\Lambda$-Moran model with selection. As an application we compute a recursion for the fixation probabilities.

The paper is organised as follows. In the next section, we will define the $\Lambda$-asymmetric Moran model and provide the generator of its frequency process. In Section \ref{sect:adaptation-selection} we will state the central coupling lemma and construct the ancestral selection graph.  Using the construction of the ancestral graph, we will consider the ancestral process and show our duality result in Section \ref{sect:ancestral_process}. A generalization of the $\Lambda$-asymmetric Moran to the case of finite measures is provided in Section \ref{sect:generalisations}, as well as some remarks on the coupling. Finally, we will discuss scaling limits in Section \ref{scaling_limits}, and fixation probabilites via Griffith's representation in Section \ref{sect:Griffiths}.

%\section{To be removed after we decide, this is a test for how things look}
%Suggestions for the typographical problem with plus and minus: Call the types $\oplus$ and $\ominus$. But keep the normal + and - in the measures, that is, write $\Lambda^{\oplus},\Lambda^{\ominus}$ and not $\Lambda^{\oplus}$ and $\lambda^{\ominus}.$ Alternative notation: $\boxplus$ and $\boxminus$ in the text.  Also we have + and - vs. $+$ and $-$, but the math versions in the text look like a hyphen. The packages MnSymbol and mathabx provide slightly different looking versions of the circle and box version, my personal favourite is MnSymbol (because it's smaller than the others, mathabx has the advantage that the symbols don't quite go up to the boundary of the box or circle, but that's hard to see anyway).

\section{$\Lambda-$asymmetric Moran model}\label{sect:Moran}
In this section, we define our main object of interest, the $\Lambda$-asymmetric Moran model. It is related to the Moran model with viability selection of Etheridge, Griffiths and Taylor \cite{EGT}, where for simplicity we restrict ourselves to the case of two types and no mutation.  More importantly, we also require our process to be consistent and exchangeable.

We consider a continuous time Moran  model with fixed population size $N.$ The two types will be denoted by $\oplus$ and $\ominus$,  we write $\tau(i,t)\in\{\oplus,\ominus\}$ for the type of individual $i\in[N]=\{1,...,N\}$ at time $t\geq 0.$ Denote by $\mathcal{M}[0,1]$ the set of finite measures and by $\mathcal{M}_1[0,1]$ the set of probability measures on $[0,1]$ equipped with the Borel sigma field.  Fix probability measures $\Lambda^{\oplus}, \Lambda^{\ominus}\in \mathcal{M}_1[0,1].$ These measures will provide the reproduction rates and the strength of the selection at a reproductive event, meaning that a fraction of the population, whose size is determined by $\Lambda^{\oplus}$ or $\Lambda^{\ominus}$, will be replaced by offspring of the reproducing individual.  Note that the restriction to probability measures is uniquely for the simplicity of the presentation. We will explain in Subsection \ref{sect:adaptation-selection} below how the construction of the model easily generalises to finite measures $\Lambda^{\oplus},\Lambda^{\ominus}$.  We stress that not only the construction of the model but all results presented below, in particular the construction of the ancestral selection graph, also hold for finite measures, with no changes to the proofs.

\begin{defn}[$\Lambda$-asymmetric Moran model]\label{def:Lambda-Moran} At rate 1, an individual chosen uniformly among the current population of size $N$ reproduces.  If at time $t$ individual $i$ reproduces and has type $\tau(i,t)\in\{\oplus,\ominus\}$, then the selective strength of the reproductive event is provided as a random variable $Y$ sampled independently of everything else from the probability measure $\Lambda^{\tau(i,t)}$.  Conditional on $Y,$ each of the $N-1$ individuals in $[N]\setminus \{i\}$ participates in the reproduction event with probability $Y,$ meaning that the participating individuals die and are replaced by offspring of the reproducing individuals, carrying the type of individual $i.$ 
\end{defn}

\begin{figure}[t]
       \centering
    \scalebox{1}{
        \begin{tikzpicture}[darkstyle/.style={circle,inner sep=0pt}]
        
        %Lines
        \draw[-,thick](0,0)--(6,0);
        \draw[-,thick](0,1)--(6,1);
        \draw[-,thick](0,2)--(6,2);
        \draw[-,thick](0,3)--(6,3);
        \draw[-,thick](0,4)--(6,4);
        \draw[-,thick](0,5)--(6,5);

       %Dots/Individuals
        \node (A1) at (1,0) [circle,scale=0.3]{} ;
        \node (A2) at (1,1) [circle,scale=0.3]{} ;
        \node (A3) at (1,2) [circle, scale=0.5, draw, fill=black]{} ;
        \node (A4) at (1,3) [circle, scale=0.3]{} ;
        \node (A5) at (1,4) [circle, scale=0.3]{} ;
        \node (A6) at (1,5) [circle,scale=0.3]{} ;
        
         \node (B1) at (3.5,0) [circle,scale=0.3]{} ;
        \node (B2) at (3.5,1) [circle,scale=0.3]{} ;
        \node (B3) at (3.5,2) [circle,scale=0.3]{} ;
        \node (B4) at (3.5,3) [circle,scale=0.5, draw, fill=black]{} ;
        \node (B5) at (3.5,4) [circle,scale=0.3]{} ;
        \node (B6) at (3.5,5) [circle,scale=0.3]{} ;
        
%        \node (C1) at (5,0) [circle,scale=0.75, draw, fill=black]{} ;
%        \node (C2) at (5,1) [circle,scale=0.75, draw, fill=black]{} ;
%        \node (C3) at (5,2) [circle,scale=0.75, draw, fill=black]{} ;
%        \node (C4) at (5,3) [circle,scale=0.75, draw, fill=black]{} ;
%        \node (C5) at (5,4) [circle,scale=0.75, draw, fill=black]{} ;
%        \node (C6) at (5,5) [circle,scale=0.75, draw, fill=black]{} ;
%        
        
        %Arrows
         \draw[->, very thick] (A3) to[bend left=75] (A6);
        \draw[->, very thick] (A3) to[bend left=45] (A5);
         \draw[->, very thick] (A3) to[bend right=45] (A2);
          \draw[->, very thick] (A3) to[bend right=60] (A1);
         
         \draw[->, very thick] (B4) to[bend right=75] (B1);
         \draw[->, very thick] (B4) to[bend right=45] (B3);
         %\draw[->, ultra thick] (B4) to[bend left=45] (B5);
        % \draw[->, ultra thick, color=gray!70!white] (B4) to[bend right=45] (B2);
         
         % \draw[->, ultra thick, color=gray!70!white] (C1) to (C2);
          %\draw[->, ultra thick, color=gray!70!white] (C1) to[bend left=45](C5);
          
          %Types
          \node () at (-0.2,0){$\bm{\ominus}$};
          \node () at (-0.2,1){$\bm{\ominus}$};
          \node () at (-0.2,2){$\bm{\oplus}$};
          \node () at (-0.2,3){$\bm{\ominus}$};
          \node () at (-0.2,4){$\bm{\ominus}$};
          \node () at (-0.2,5){$\bm{\ominus}$};
          
           \node () at (6.2,0){$\bm{\ominus}$};
          \node () at (6.2,1){$\bm{\oplus}$};
          \node () at (6.2,2){$\bm{\ominus}$};
          \node () at (6.2,3){$\bm{\ominus}$};
          \node () at (6.2,4){$\bm{\oplus}$};
          \node () at (6.2,5){$\bm{\oplus}$};
        \end{tikzpicture}
        }   
      \caption{A realisation of the $\Lambda$-asymmetric frequency process. Arrows point from the reproducing individual to their offspring.  In the first reproductive event, a type $\oplus$ reproduces, in the second one a type $\ominus$ individual. The role of selection in this construction will be discussed later in Section \ref{sect:adaptation-selection}, cf.  also Figure \ref{fig:frequency_Lambda}.}
      \label{fig:frequency}
\end{figure}
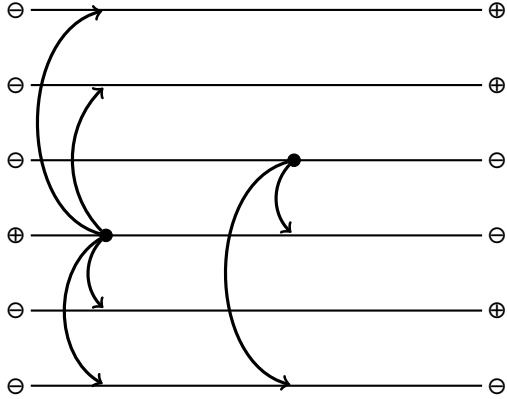

Our model is closely related to the \emph{Moran model with viability selection} introduced in \cite{EGT}. There, individuals reproduce at fixed rate $\lambda$ and  produce a number of offspring, among which only some of the children survive to maturity.  The probability of an individual of type $\tau$ to have $b$ mature offspring is given as $p_{\tau b}=\sum_{a=b}^{N-1} r_av_{\tau ab},$ with $r_a$ the number of children, and $v_{\tau ab}$ the probability that $b$ out of $a$ children survive to maturity. In particular $v_{\tau ab}=\int_{[0,1]}\binom{a}{b}p^b(1-p)^{a-b}\nu_\tau(dp)$ for some probability measure $\nu_\tau$ may be considered.  This corresponds to our model if we set $\Lambda^\tau=\lambda\nu_\tau$ and $r_{N-1}=1.$ Thus the model of \cite{EGT} is a priori more general, however, their methods require the process to be in stationarity, and thus necessarily needs to incorporate mutations.  Our construction doesn't need any stationarity assumption and works for any initial conditions.  Further, our construction always leads to a consistent and exchangeable process. In the model of \cite{EGT} this is not the case if one takes for example $r_a=\delta_k$ for some fixed $k\in\N$ for all $a$, which would violate consistency.  In view of \cite{MS}, see also \cite{JKR}, we expect our model to cover all consistent and exchangeable Moran models with coordinated selection in terms of two finite $\Lambda$-measures.  Exchangeability plays a crucial role in the construction of the ancestral selection graph for our model below.  In our notation the $\Lambda-$asymmetric Moran model doesn't directly lead to an interpretation in terms of viability selection. On the other hand,  we can at least implicitly incorporate various selective mechanisms in our finite measures $\Lambda^\tau.$ See also Remark  \ref{rem:cases_of_reproduction} below for an interpretation of the $\Lambda$-measures in some special cases.

The ancestral process of the $\Lambda$-asymmetric Moran model that will be derived in Section \ref{sect:ancestral_process} is similar to the one studied in \cite{BBDP}, where the probability of fixation for a related $\Lambda-$Moran model involving Beta measures is calculated under stationarity assumptions.  A notable difference is the fact that the ancestral process of \cite{BBDP} has linear branching rates,which is not in general the case for our model and for the model of \cite{EGT}. A similar process also appears for Moran models in a one-sided random environment, \cite{CV, GSW}.  In \cite{BEH} an interesting construction of the ancestral selection graph for a Moran model with frequency dependent selection is provided, leading to a duality with respect to the same function as for our model, see Corollary \ref{cor:duality}.  Moran models with other types of $\Lambda$-reproduction undergoing selection but with a different role of the $\Lambda$ measure have also been investigated by Schweinsberg  \cite{Schw1, Schw2} and studied further by Bah and Pardoux and Ged \cite{BP, Ged}.

We denote by
\[X^N_t=\frac{1}{N}\sum_{i=1}^N 1_{\{\tau((t,i))=\ominus\}},\qquad t\ge0,\]
the frequency at time $t$ of individuals of type $\ominus$, which will later be the less fit type.  Clearly the process $X^N=(X_t^N)_{t\geq 0}$  is a piecewise continuous Markov chain with state space $S_N:=\{0, 1/N,...,(N-1)/N, 1\}$. Its transitions are given by
\[
x\mapsto \begin{cases} x+\frac{k}{N}&  \text{ at rate } x\int_0^1 \binom{(1-x)N}{k} y^k (1-y)^{(1-x)N-k}\Lambda^{\ominus}(dy), \quad k=1,...,(1-x)N \\
x-\frac{k}{N}& \text{ at rate } (1-x) \int_0^1 \binom{xN}{k}y^k(1-y)^{xN-k}\Lambda^{\oplus}(dy),\quad k=1,...,xN .
\end{cases}\]

Observe that atoms at 0 of $\Lambda^\oplus$ and $\Lambda^\ominus$ don't influence these rates. In words, upon a reproduction event of a type $\ominus$ individual, which happens at rate $x$, the strength $y$ of the reproduction is determined according to the measure $\Lambda^{\ominus}$. Then, independently with probability $y,$ each of the $N-1$ non-reproducing individuals dies and is replaced by an offspring of type $\ominus$.  With probability $1-y$ it remains. Only a replacement of a type $\oplus$ individual by an offspring of type $\ominus$ leads to an increase in type $\ominus$ individuals, thus there are $(1-x)N$ individuals that may swap type from $\oplus$ to $\ominus$ at such an event. The number of individuals switching type from $\oplus$ to $\ominus$ is thus binomial with success probability $y$ and $N(1-x)$ trials.  A reproduction of a type $\oplus$ individuals works vice versa.  See Figure \ref{fig:frequency} for an example.

Hence,  the generator of the frequency process of the $\Lambda$-asymmetric Moran model acts on bounded measurable functions $f:[0,1]\to[0,1]$ as 
\begin{align}\label{gen_frequency}
	\mathcal{B}^Nf(x)
	&= x\E\left[f\left(x+\frac{1}{N}\text{Binom}(N(1-x),Y^{\ominus})\right)-f(x)\right]\notag\\
	&+(1-x)\E\left[f\left(x-\frac{1}{N}\text{Binom}(Nx,Y^{\oplus})\right)-f(x)\right].
\end{align}

Here, $x\in S_N$, %[N_0]:=\{0, 1/N,...,(N-1)/N, 1\},$ 
and the expectation is taken with respect to the random variables $Y^{\ominus}$ resp. $Y^{\oplus}$, which are distributed according to the probability measures $\Lambda^{\ominus}$ resp. $\Lambda^{\oplus}$.

\section{Stochastic ordering and the $\Lambda$-asymmetric ancestral selection graph}\label{sect:adaptation-selection}

The next aim is to construct an ancestral selection graph corresponding to the $\Lambda$-asymmetric Moran model defined in the previous section.  We will give a graphical representation in terms of Poisson Point processes which forward in time provides a construction of the $\Lambda$-asymmetric Moran model, and backward in time gives the ancestry of a sample of such a population. Our construction differs from the graphical representation of Etheridge, Griffiths and Taylor in several ways. In particular,  the construction works for any assignment of types to the individuals, and it doesn't need an invariant distribution. It thus provides a pathwise construction of the duality found in \cite{EGT}, see Corollary \ref{cor:duality} below.

\medskip

As a trade-off,  we need the measures $\Lambda^{\ominus}$ and $\Lambda^{\oplus}$ to be stochastically ordered. This assumption will allow us to couple the two measures, a fact that will play a crucial role in the construction.  This order is readily interpreted in terms of the selective advantage of the reproduction mechanisms: An individual of type $\oplus$ tends to have more offspring at a selective event than an individual of type $\ominus$. See also the discussion in Section \ref{sect:generalisations} below. Moreover, it can easily be generalised to finite measures.

We recall that two probability measures $\mu_0$ and $\mu_1$ on $[0,1]$ are \emph{stochastically ordered,} we write $\mu_0\leq\mu_1$, if for every bounded increasing function $f:[0,1]\to \R$ we have
\[\int_{[0,1]}f(x)\mu_0(dx)\leq \int_{[0,1]}f(x)\mu_1(dx).\]

This provides a partial order on the set of probability measures on $[0,1]$. Note that by the usual measure theoretic approximation in terms of step functions, this is equivalent to requiring $\mu_0[x,1]\leq \mu_1[x,1]$ for every $x\in[0,1].$

The crucial step in our construction is the following coupling lemma. For stochastically ordered measures $\Lambda^{\ominus}$ and $\Lambda^{\oplus}$, in the construction of the $\Lambda-$asymmetric Moran model we can equivalently consider a Moran model constructed from just one measure $\Lambda$ that contains the information of a particular coupling of the pair $(\Lambda^{\ominus}, \Lambda^{\oplus}).$ 

\begin{lem}[Coupling lemma]\label{descom_measure_new} %\label{descom_measure} 
	Let $\Delta=\{(y,z)\in[0,1]^2:y+z\in[0,1]\}$ and consider two probability measures $\Lambda^{\oplus},\Lambda^{\ominus}$ on $[0,1]$ such that $\Lambda^{\ominus}\leq\Lambda^{\oplus}$ in the stochastic partial order.  Then there exists a finite measure  $\Lambda$ on $\Delta$ such that for all $A,B \in \mathcal{B}([0,1]),$
\[ \Lambda^{\ominus}(A)=\Lambda(\{(y,z): y\in A\}) \quad  \text { and } \quad \Lambda^{\oplus}(B)=\Lambda(\{(y,z): y+z\in B\})\]
hold. 
\end{lem}

\begin{proof}
By stochastic ordering we know that $\Lambda^{\ominus}[0,x]-\Lambda^{\oplus}[0,x]\geq0$ for any $x\in[0,1]$. Therefore by a version of Strassen's Theorem (see (3) in \cite{Lindvall1999}) there exists two random variables $Z^{\ominus}$ and $Z^{\oplus}$ on a probability space such that $Z^{\ominus}\sim \Lambda^{\ominus}(\cdot)$ and $Z^{\oplus}\sim \Lambda^{\oplus}(\cdot),$
and $Z^{\ominus}\leq Z^{\oplus}$ a.s.  Here $\sim$ means that the random variable is distributed according to the respective probability measure. Moreover, the joint distribution of $(Z^{\ominus},Z^{\oplus})$, is given as 
\begin{align*}
	\mathbb{P}\left((Z^{\ominus},Z^{\oplus})\in A\right)=\mathbb{P}\left((F_{\ominus}^{-1}(U),F_{\oplus}^{-1}(U))\in A\right),\qquad A\in\mathcal{B}([0,1]^2),
\end{align*}
where $F_{\ominus}$ (resp. $F_{\oplus}$) is the distribution function of the random variable $Z^{\ominus}$ (resp. $Z^{\oplus}$) and $U$ is a unifom random variable on $[0,1]$.
Hence,  by setting
\begin{align*}
	\Lambda(A):=\mathbb{P}\left((Z^{\ominus},(Z^{\oplus}-Z^{\ominus}))\in A\right),\qquad \text{for any $A\in\mathcal{B}(\Delta)$,}
\end{align*}
we note that for any $B\in \mathcal{B}([0,1])$ we have
$	\Lambda(B\times[0,1])=\mathbb{P}\left((Z^{\ominus},Z^{\oplus}-Z^{\ominus})\in B\times[0,1]\right)=\mathbb{P}(Z^{\ominus}\in B)=\Lambda^{\ominus}(B)$
and
$\Lambda\left(\{(y,z):y+z\in A\}\right)=\mathbb{P}\left(Z^{\ominus}+Z^{\oplus}-Z^{\ominus}\in A\right)=\mathbb{P}\left(Z^{\oplus}\in A\right)=\Lambda^{\oplus}(A).$
\end{proof}

By the coupling lemma, the probability measure $\rho$ on $[0,1]^2$ defined by 
	\begin{align*}
		\rho(A\times B)=\Lambda(\{(y,z): y\in A, y+z\in B\}),\qquad A,B\in\mathcal{B}([0,1]),
	\end{align*} 
	is a monotone coupling of $\Lambda^{\ominus}$ and $\Lambda^{\oplus}$ (that is $\rho\{(y,z):y>z\}=0$). We sometimes (by a slight abuse of notation) refer to $\Lambda$ as the \emph{selective coupling} of $\Lambda^{\ominus}$ and $\Lambda^{\oplus}$. We will usually apply it in the form
	\begin{align}\label{eq:decomp-int}
	\int_{\Delta}f(y)\Lambda(dy,dz)=\int_{[0,1]}f(y)\Lambda^{\ominus}(dy),\ \text{and} \ \int_{\Delta}f(y+z)\Lambda(dy,dz)=\int_{[0,1]}f(z)\Lambda^{\oplus}(dz)
	\end{align}

for measurable functions $f:[0,1]\mapsto[0,1]$. In Section \ref{sect:generalisations} we explain how to apply this Lemma in the case of finite measures $\Lambda^{\oplus}$ and $\Lambda^{\ominus}$ and provide a small caveat on equation \eqref{eq:decomp-int} in this case. Moreover, we provide a brief discussion of uniqueness questions, and an interpretation in terms of optimal transport.

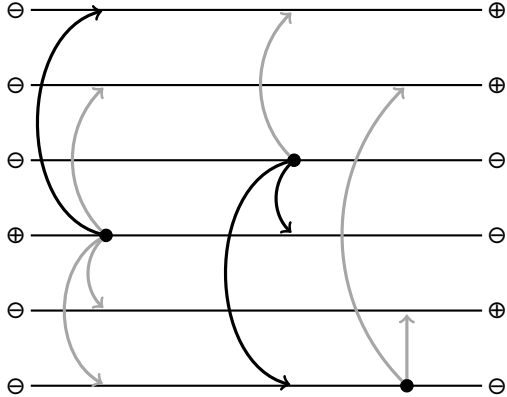
\begin{figure}[t]
    \centering
    \scalebox{1}{
        \begin{tikzpicture}[darkstyle/.style={circle,inner sep=0pt}]
        
        %Lines
        \draw[-,thick](0,0)--(6,0);
        \draw[-,thick](0,1)--(6,1);
        \draw[-,thick](0,2)--(6,2);
        \draw[-,thick](0,3)--(6,3);
        \draw[-,thick](0,4)--(6,4);
        \draw[-,thick](0,5)--(6,5);

       %Dots/Individuals
        \node (A1) at (1,0) [circle,scale=0.3]{} ;
        \node (A2) at (1,1) [circle,scale=0.3]{} ;
        \node (A3) at (1,2) [circle, scale=0.5, draw, fill=black]{} ;
        \node (A4) at (1,3) [circle, scale=0.3]{} ;
        \node (A5) at (1,4) [circle, scale=0.3]{} ;
        \node (A6) at (1,5) [circle,scale=0.3]{} ;
        
         \node (B1) at (3.5,0) [circle,scale=0.3]{} ;
        \node (B2) at (3.5,1) [circle,scale=0.3]{} ;
        \node (B3) at (3.5,2) [circle,scale=0.3]{} ;
        \node (B4) at (3.5,3) [circle,scale=0.5, draw, fill=black]{} ;
        \node (B5) at (3.5,4) [circle,scale=0.3]{} ;
        \node (B6) at (3.5,5) [circle,scale=0.3]{} ;
        
        \node (C1) at (5,0) [circle,scale=0.5, draw, fill=black]{} ;
        \node (C2) at (5,1) [circle,scale=0.3]{} ;
        \node (C3) at (5,2) [circle,scale=0.3]{} ;
        \node (C4) at (5,3) [circle,scale=0.3]{} ;
        \node (C5) at (5,4) [circle,scale=0.3]{} ;
        \node (C6) at (5,5) [circle,scale=0.3]{} ;

        %Arrows
         \draw[->, very thick] (A3) to[bend left=75] (A6);
         \draw[->, very thick, color=gray!70!white] (A3) to[bend left=45] (A5);
         \draw[->, very thick, color=gray!70!white] (A3) to[bend right=45] (A2);
          \draw[->, very thick, color=gray!70!white] (A3) to[bend right=60] (A1);
         
         \draw[->, very thick] (B4) to[bend right=75] (B1);
         \draw[->, very thick] (B4) to[bend right=45] (B3);
        % \draw[->, ultra thick] (B4) to[bend left=45] (B5);
         \draw[->, very thick, color=gray!70!white] (B4) to[bend left=45] (B6);
         
          \draw[->, very thick, color=gray!70!white] (C1) to (C2);
          \draw[->, very thick, color=gray!70!white] (C1) to[bend left=45](C5);
          
          %Types
          \node () at (-0.2,0){$\bm{\ominus}$};
          \node () at (-0.2,1){$\bm{\ominus}$};
          \node () at (-0.2,2){$\bm{\oplus}$};
          \node () at (-0.2,3){$\bm{\ominus}$};
          \node () at (-0.2,4){$\bm{\ominus}$};
          \node () at (-0.2,5){$\bm{\ominus}$};
          
           \node () at (6.2,0){$\bm{\ominus}$};
          \node () at (6.2,1){$\bm{\oplus}$};
          \node () at (6.2,2){$\bm{\ominus}$};
          \node () at (6.2,3){$\bm{\ominus}$};
          \node () at (6.2,4){$\bm{\oplus}$};
          \node () at (6.2,5){$\bm{\oplus}$};
        \end{tikzpicture}
        }   
      \caption{The realisation of the  $\Lambda$-asymmetric frequency process from Figure \ref{fig:frequency} now in terms of the coupling construction.  Neutral arrows are black, selective arrows grey. Individuals of type $\oplus$ may reproduce through any arrow, individuals of type $\ominus$ only through neutral arrows. Therefore some of the arrows and the last reproductive event weren't present in Figure \ref{fig:frequency}, where the measures applied to determine the reproduction depended on the type of the reproducing individuals.}\label{fig:frequency_Lambda}
\end{figure}

With this coupling, we will construct the $\Lambda$-asymmetric ancestral selection graph.  In order to make the construction more transparent, we first provide an alternative description of the $\Lambda-$asymmetric Moran model, using the stochastic partial order. Assume that the measures $\Lambda^{\oplus}$ and $\Lambda^{\ominus}$ satisfy $\Lambda^{\ominus}\leq\Lambda^{\oplus}$. Let $\Lambda$ be its selective coupling in the sense of Lemma \ref{descom_measure_new}.  This means that the random variable $Y$ that was sampled in Definition \ref{def:Lambda-Moran} according to either $\Lambda^{\oplus}$ or $\Lambda^{\ominus},$ depending on the type of the reproducing individual, can now be sampled instead as a pair $(Y,Z)$ according to $\Lambda$ from $\Delta,$ such that individuals are replaced with probability $Y$ if a type $\ominus$ individual reproduces, and individuals are replaced with probability $Y+Z$ if a type $\oplus$ individual reproduces.  The rates of the frequency process can thus be rewritten as (cf. \eqref{eq:decomp-int})
\[
x\mapsto \begin{cases} x+\frac{k}{N}&  \text{ at rate } x\int_{\Delta} \binom{(1-x)N}{k} y^k (1-y)^{(1-x)N-k}\Lambda(dy,dz), \quad k=1,...,(1-x)N \\
	x-\frac{k}{N}& \text{ at rate } (1-x) \int_{\Delta}\binom{xN}{k}(y+z)^k(1-(y+z))^{xN-k}\Lambda(dy,dz), \quad k=1,...,Nx.
\end{cases}\]
The generator \eqref{gen_frequency} becomes 
\begin{align}
	\mathcal{B}^Nf(x)
	&= x\int_{\Delta}\E\left[f\left(x+\frac{1}{N}\text{Binom}(N(1-x),y)\right)-f(x)\right]\Lambda(dy,dz)\notag\\
	&+(1-x)\int_{\Delta}\E\left[f\left(x-\frac{1}{N}\text{Binom}(Nx,y+z)\right)-f(x)\right]\Lambda(dy,dz).
\end{align}		

We are now ready to define the central object of this paper, which is the ancestral selection graph for the $\Lambda-$asymmetric Moran model. We give a construction in the spirit of Neuhauser and Krone in terms of a Poisson process driven by the selective coupling $\Lambda.$

\begin{defn}[$\Lambda-$asymmetric ancestral selection graph, ASG]\label{def:ancestral_selection_graph} Consider a Poisson processes $M^N$ with values in $\mathbb{R}_+\times [N]\times \Delta\times[0,1]^N$ and intensity measure $dt\times dm\times \Lambda(dy,dz)\times du_1\times du_2...\times du_N$, where $dm$ denotes the uniform measure
%\blue{[For the proof of Proposition \ref{gen_freq_ASG} I guess we need that $m$ is the uniform measure on $[N]$ not the counting measure? Please see my comment in \eqref{ito_1}]}
on $[N]$. Each point $(t,i)\in\mathbb{R}\times [N]$, represents the $i$-th individual alive at time $t$. 
	\begin{itemize}
		\item  We say that at time $t$ there is a \emph{neutral arrow} between $i$ and $j$ if there is a point \linebreak $(t,i,y,z,u_1,u_2,...,u_N)\in M^N$ such that $u_j\in[0,y]$.
		\item We say that at time $t$ there is a \emph{selective arrow} between $i$ and $j$ if there is a point \linebreak $(t,i,y,z,u_1,u_2,...,u_N)\in M^N$ such that $u_j\in [y,y+z]$.
	\end{itemize}
	The ancestral selection graph is then given by $(\R_+\times [N], M^N).$	
\end{defn}

In the graphical representation of the ancestral selection graph,  as usual, individuals are represented by lines, and reproductive events are denoted by arrows, where the arrow starts at the reproducing individual, and the tip points to the line of (potential) offspring.  Here we represent the reproducing individuals by filled dots, the neutral arrows are black and the selective arrows are grey, see Figure \ref{fig:frequency_Lambda}. Neutral and selective arrows may occur at the same time.  Observe that formally we can have both neutral and selective arrows from $i$ to itself, but they won't have an effect on the frequency process. The ancestral selection graph may be constructed using any coupling that satisfies the conditions of Lemma \ref{descom_measure_new}. In section \ref{sect:generalisations} we show that the expected number of selective arrows will be the same under any such coupling, while the Strassen coupling applied in the proof of the coupling Lemma minimises the variance of the number of these arrows.

From the graphical representation and the Poisson point process construction we can derive the frequency process.  Heuristically, at an arrival of the Poisson process, the reproducing individual $i$ is chosen uniformly at random, and the reproductive strength is determined by the measure $\Lambda.$ The uniform values $u_j$ then determine whether individual $j$ participates in a neutral or selective reproduction event. After introducing types, type $\oplus$ will be propagated (that is, the individual at the tip ot the arrow will be replaced) both through neutral and selective arrows, while type $\ominus$ will be propagated only through neutral arrows.  Hence we can construct the frequency process from the ASG by distributing types at time $t=0,$ and propagating them forward in time according to the arrows on the ASG.  

\begin{prop}\label{prop:frequency_equal}
The frequency process constructed from the $\Lambda$-asymmetric ancestral selection graph by propagating type $\oplus$ through both neutral and selective arrows, and type $\ominus$ through neutral arrows only, is the same as the frequency process of the $\Lambda$-asymmetric Moran model.
\end{prop}

We will give a more formal definition of the frequency process and a proof of this result in section \ref{sect:ancestral_process} below.

\section{Ancestral process, sampling function and duality}\label{sect:ancestral_process}
In this section we will formalize the connection between the $\Lambda$-asymmetric Moran model and the ancestral selection graph from Definition \ref{def:ancestral_selection_graph}. We will also define the ancestral line counting process and state its pathwise sampling duality with the frequency process.

We start by introducing suitable forward and backward filtrations.  For any subset of individuals $S\subset[N]$ and any time $T\in \R_+$ we will write $S_{T}=\{(T,i): i\in S\}.$ We think of $S_T$ as a (random) sample of individuals taken at time $T$.  We will follow the construction detailed in  \cite{GSW} and define the forward and backward filtrations of our ancestral selection graph.

\begin{defn}
Let $(\mathbb{R}_+\times [N], M^N)$ be the $\Lambda$-asymmetric ancestral selection graph, and let $S_T$ be a sample as defined above. 
	\begin{itemize}
		\item  Let $\mathcal{F}^T_t$ denote the sigma algebra generated by $S_T$ and the restriction of $M^N$ to $[T,T+t]$. Then $\{\mathcal{F}^T_t\}_{t\geq 0}$ is the \emph{forward filtration}.
		\item Let $\mathcal{P}^T_t$ denote the sigma algebra generated by $S_T$ and the restriction of $M^N$ to $[T,T-t]$. Then $\{\mathcal{P}^T_t\}_{t\geq 0}$ is the \emph{backward filtration}.
	\end{itemize}
\end{defn}

The ancestral selection graph is particularly useful as a construction that can be used both forward and backward in time. In order to utilize this, we introduce the notion of ancestral path in the ancestral selection graph, which will go backward in time. 

\begin{defn}\label{def:ancestral_path}
Let $t,s\in \R_+, i,j\in[N].$	An \emph{ancestral path} going from $(t+s,i)$ to $(t,j)$ is the graph of a c\`adl\`ag function $f:[t,t+s]\mapsto [N]$ such that:
	\begin{enumerate}
		\item $f(t)=j$, $f(t+s)=i$. 
		\item If $f(u^-)=k\neq f(u)=l$ then there is an arrow (neutral or selective) in the ancestral selection graph $M^N$ from $(u,k)$ to $(u,l)$.
		\item If there is an arrow from $(u,k)$ to $(u,l)$, then $f(u^-)\neq l$. 
	\end{enumerate}
	We say that an individual $(t,j)$ is a \emph{potential ancestor} of $(t+s,i)$ if there exists an ancestral path going from $(t+s,i)$ to $(t,j)$ and in this case we write $(t,j)\sim(t+s,i)$. 
\end{defn}
Graphically,  the concept of potential ancestor and of ancestral paths is described in Figure \ref{fig:ancestral}.  

\begin{figure}[t]

{ {\begin{minipage}[t]{240pt}
  % \centering
    \scalebox{1}{
        \begin{tikzpicture} %[darkstyle/.style={inner sep=0pt}]
        
        %Lines
        \draw[-,thick, color=gray!70!white](0,0)--(3.5,0);
        \draw[-,thick, color=gray!70!white](5,0)--(6,0);
        \draw[-,ultra thick, color=black](3.5,0)--(5,0);
        \draw[-,ultra thick, color=black](0,1)--(6,1);
        \draw[-,ultra thick](3.5,2)--(6,2);
       % \draw[-,thick, color=gray!70!white](0,3)--(3.5,3);
        \draw[-,ultra thick, color=black](0,3)--(6,3);
        \draw[-,ultra thick, color=black](0,2)--(1,2);
      \draw[-,thick, color=gray!70!white](1,2)--(3.5,2);
        \draw[-,thick, color=gray!70!white](0,4)--(6,4);
        \draw[-,thick, color=gray!70!white](0,5)--(6,5);

       %Dots/Individuals
        \node (A1) at (1,0) [circle,scale=0.3]{} ;
        \node (A2) at (1,1) [circle,scale=0.3]{} ;
        \node (A3) at (1,2) [circle, scale=0.5, draw, fill=black]{} ;
        \node (A4) at (1,3) [circle, scale=0.3]{} ;
        \node (A5) at (1,4) [circle, scale=0.3]{} ;
        \node (A6) at (1,5) [circle,scale=0.3]{} ;
        
         \node (B1) at (3.5,0) [circle,scale=0.3]{} ;
        \node (B2) at (3.5,1) [circle,scale=0.3]{} ;
        \node (B3) at (3.5,2) [circle,scale=0.3]{} ;
        \node (B4) at (3.5,3) [circle,scale=0.5, draw, fill=black]{} ;
        \node (B5) at (3.5,4) [circle,scale=0.3]{} ;
        \node (B6) at (3.5,5) [circle,scale=0.3]{} ;
        
        \node (C1) at (5,0) [circle,scale=0.5, draw, fill=black]{} ;
        \node (C2) at (5,1) [circle,scale=0.3]{} ;
        \node (C3) at (5,2) [circle,scale=0.3]{} ;
        \node (C4) at (5,3) [circle,scale=0.3]{} ;
        \node (C5) at (5,4) [circle,scale=0.3]{} ;
        \node (C6) at (5,5) [circle,scale=0.3]{} ;

        %Arrows
         \draw[->, very thick] (A3) to[bend left=75] (A6);
         \draw[->, very thick, color=gray!70!white] (A3) to[bend left=45] (A5);
         \draw[->, very thick, color=gray!70!white] (A3) to[bend right=45] (A2);
          \draw[->, very thick, color=gray!70!white] (A3) to[bend right=60] (A1);
         
         \draw[->, very thick] (B4) to[bend right=75] (B1);
         \draw[->, very thick] (B4) to[bend right=45] (B3);
        % \draw[->, ultra thick] (B4) to[bend left=45] (B5);
         \draw[->, very thick, color=gray!70!white] (B4) to[bend left=45] (B6);
         
          \draw[->, very thick, color=gray!70!white] (C1) to (C2);
          \draw[->, very thick, color=gray!70!white] (C1) to[bend left=45](C5);
          
          %Ancestry
         % \node() at (6.2,-0.05){*};
           \node() at (6.2,1.95){*};
            \node() at (6.2,2.95){*};
             \node() at (6.2,0.95){*};
          %Types
%          \node () at (-0.2,0){$\bm{\ominus}$};
%          \node () at (-0.2,1){$\bm{\ominus}$};
%          \node () at (-0.2,2){$\bm{\oplus}$};
%          \node () at (-0.2,3){$\bm{\ominus}$};
%          \node () at (-0.2,4){$\bm{\ominus}$};
%          \node () at (-0.2,5){$\bm{\ominus}$};
%          
%           \node () at (6.2,0){$\bm{\ominus}$};
%          \node () at (6.2,1){$\bm{\oplus}$};
%          \node () at (6.2,2){$\bm{\ominus}$};
%          \node () at (6.2,3){$\bm{\ominus}$};
%          \node () at (6.2,4){$\bm{\oplus}$};
%          \node () at (6.2,5){$\bm{\oplus}$};
        \end{tikzpicture}
        }   
         \end{minipage}}}
  {\begin{minipage}[t]{240pt}
   \scalebox{1}{\hspace{-0.6cm}
        \begin{tikzpicture} %[darkstyle/.style={inner sep=0pt}]
        
        %Lines
        \draw[-,thick, color=gray!70!white](0,0)--(3.5,0);
        \draw[-,ultra thick, color=black](3.5,0)--(5,0);
        \draw[-,thick, color=gray!70!white](5,0)--(6,0);
        \draw[-,thick, color=gray!70!white](0,1)--(6,1);
        \draw[-,ultra thick, color=black](0,2)--(1,2);
      \draw[-,thick, color=gray!70!white](1,2)--(6,2);
        %\draw[-,ultra thick](3.5,2)--(6,2);
       % \draw[-,thick, color=gray!70!white](0,3)--(3.5,3);
        \draw[-,ultra thick,  color=black](0,3)--(3.5,3);
         \draw[-,thick, color=gray!70!white](3.5,3)--(6,3);
        \draw[-,ultra thick, color=black](0,4)--(6,4);
        \draw[-,thick, color=gray!70!white](0,5)--(6,5);

       %Dots/Individuals
        \node (A1) at (1,0) [circle,scale=0.3]{} ;
        \node (A2) at (1,1) [circle,scale=0.3]{} ;
        \node (A3) at (1,2) [circle, scale=0.5, draw, fill=black]{} ;
        \node (A4) at (1,3) [circle, scale=0.3]{} ;
        \node (A5) at (1,4) [circle, scale=0.3]{} ;
        \node (A6) at (1,5) [circle,scale=0.3]{} ;
        
         \node (B1) at (3.5,0) [circle,scale=0.3]{} ;
        \node (B2) at (3.5,1) [circle,scale=0.3]{} ;
        \node (B3) at (3.5,2) [circle,scale=0.3]{} ;
        \node (B4) at (3.5,3) [circle,scale=0.5, draw, fill=black]{} ;
        \node (B5) at (3.5,4) [circle,scale=0.3]{} ;
        \node (B6) at (3.5,5) [circle,scale=0.3]{} ;
        
        \node (C1) at (5,0) [circle,scale=0.5, draw, fill=black]{} ;
        \node (C2) at (5,1) [circle,scale=0.3]{} ;
        \node (C3) at (5,2) [circle,scale=0.3]{} ;
        \node (C4) at (5,3) [circle,scale=0.3]{} ;
        \node (C5) at (5,4) [circle,scale=0.3]{} ;
        \node (C6) at (5,5) [circle,scale=0.3]{} ;

        %Arrows
         \draw[->, very thick] (A3) to[bend left=75] (A6);
         \draw[->, very thick, color=gray!70!white] (A3) to[bend left=45] (A5);
         \draw[->, very thick, color=gray!70!white] (A3) to[bend right=45] (A2);
         \draw[->, very thick, color=gray!70!white] (A3) to[bend right=60] (A1);
         
         \draw[->, very thick] (B4) to[bend right=75] (B1);
         \draw[->, very thick] (B4) to[bend right=45] (B3);
        % \draw[->, ultra thick] (B4) to[bend left=45] (B5);
         \draw[->, very thick, color=gray!70!white] (B4) to[bend left=45] (B6);
         
          \draw[->, very thick, color=gray!70!white] (C1) to (C2);
          \draw[->, very thick, color=gray!70!white] (C1) to[bend left=45](C5);

          %Types
%          \node () at (-0.2,0){$\bm{\ominus}$};
%          \node () at (-0.2,1){$\bm{\ominus}$};
          \node () at (-0.2,2){$\bm{\oplus}$};
          \node () at (-0.2,3){$\bm{\ominus}$};
         \node () at (-0.2,4){$\bm{\ominus}$};
%          \node () at (-0.2,5){$\bm{\ominus}$};
%          
%           \node () at (6.2,0){$\bm{\ominus}$};
%          \node () at (6.2,1){$\bm{\oplus}$};
%          \node () at (6.2,2){$\bm{\ominus}$};
%          \node () at (6.2,3){$\bm{\ominus}$};
          \node () at (6.2,4){$\bm{\oplus}$};
%          \node () at (6.2,5){$\bm{\oplus}$};
        \end{tikzpicture}
        }   
  
    \end{minipage}}

      \caption{On the left: The ancestral line process of a sample of three individuals (marked by an asterisk),  using the realisation of the $\Lambda$-asymmetric ancestral selection graph from Figure \ref{fig:frequency_Lambda} and going backward in time (right to left).  Lines of potential ancestors are colored in black.  At a reproductive event, lines from individuals hit by a black (neutral) arrow are alway merged with the reproducing line at the origin of the arrow, where for those hit by a grey (selective) arrow, both lines are continued. On the right: The ancestry of one individual. The type of this individual is $\oplus$ if one potential ancestor has type $\oplus$, and $\ominus$ otherwise.}\label{fig:ancestral}

\end{figure}
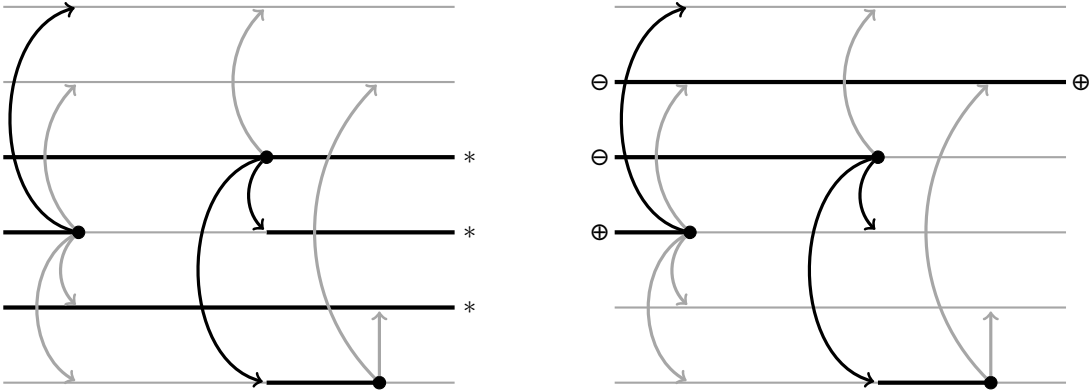

Note that, as opposed to the Moran model, an individual in this model can have many potential ancestors at any given time in its past.  Moreover, recall that we introduced neutral and selective arrows, which correspond to the measures $\Lambda^{\oplus}$ and $\Lambda^{\ominus}$ in our original definition of the $\Lambda$-asymmetric Moran model, where individuals of type $\oplus$ have a selective advantage by reproducing according to $\Lambda^{\oplus}$ and individuals of type $\ominus$ reproduce according to $\Lambda^{\ominus}$, with $\Lambda^{\ominus}\leq \Lambda^{\oplus}.$ This means that if we assign type $\ominus$ to all individuals of a sample $S_T$ taken at time $T$, and type $\oplus$ to the individuals of $S_T^c$,  then an individual $i$ at time $T+s$ is of type $\oplus$ if and only if there exists an ancestral path from  $(T+s,i)$ to at least one individual in $S^c_T.$  In this case,  we write $\tau((s,i))=\oplus$. Otherwise, if all the potential ancestors of $(T+s,i)$ at time $T$ belong to $S_T$, we write $\tau((T+s,i))=\ominus$. 

We therefore formally define the \emph{frequency process in the $\Lambda-$asymmetric ancestral selection graph}
by assigning types $\ominus,\oplus$ to all individuals at time $T>0,$ and denote by $X^{N,T}=\{X^{N,T}_t:t\geq T\}$ the frequency of type $\ominus$ individuals at time $t>T$.  Thus,
$$ X^{N,T}_t=\frac{1}{N}\sum_{i=1}^N 1_{\{\tau((t,i))=\ominus\}},\qquad t\geq T,$$
where $\tau(t,i)$ is constructed as explained above using the ancestral lines of the ASG. By construction, it is straightforward to check that the process $(X_t^{N,T})_{t\geq T}$ satisfies
\begin{align}\label{sde_ASG_1}	X^{N,T}_t&=X^{N,T}_T+\sum_{i=1}^N\int_T^t\int_{\Delta}\int_{[0,1]^N}1_{\{\tau(s,i)=\ominus\}}\frac{1}{N}\sum_{j=1}^N1_{\{\tau(s,j)=\oplus,u_j\leq y\}}M^N(di,ds,dy,dz,du)\notag\\
	&-\sum_{i=1}^N\int_T^t\int_{\Delta}\int_{[0,1]^N}1_{\{\tau(s,i)=\oplus\}}\frac{1}{N}\sum_{j=1}^N1_{\{\tau(s,j)=\ominus,u_j\leq y+z\}}M^N(di,ds,dy,dz,du), \quad t\geq T.
\end{align}

\begin{prop}\label{gen_freq_ASG}
Fix $x\in[0,1]$ and assume that $X^{N,T}_T=x$, then the process $X^{N,T}$ is a $\{\mathcal{F}^T_t\}_{t\geq 0}$ measurable continuous-time Markov chain with values in $S_N$, and its infinitesimal generator $\mathcal{B}^{N,T}$ is given for any $f\in\mathcal{C}^2([0,1])$ by
\begin{align}\label{gen_ASG_1}
\mathcal{B}^{N,T}f(x)&=x\int_{\Delta}\E\left[f\left(x+\frac{1}{N} \textup{Binom}(N(1-x),y)\right)-f(x)\right]\Lambda(dy,dz)\notag\\
&+(1-x)\int_{\Delta}\E\left[f\left(x-\frac{1}{N}\textup{Binom}(Nx,y+z)\right)-f(x)\right]\Lambda(dy,dz). %,\qquad x\in[0,1].
\end{align}
\end{prop}
\begin{proof}
The first statement of the proposition follows from \eqref{sde_ASG_1}. Now, using \eqref{sde_ASG_1} and an application of It\^o's formula we obtain for any $f\in\mathcal{C}^2([0,1])$ and $t\geq T$
\begin{align*}
&f(X^{N,T}_t)=f(X_T^{N,T})\notag\\&+\sum_{i=1}^N\int_T^t\int_{\Delta}\int_{[0,1]^N}\Bigg[f\left(X_{s-}^{N,T}+1_{\{\tau(s,i)=\ominus\}}\frac{1}{N}\sum_{j=1}^N1_{\{\tau(s,j)=\oplus,u_j\leq y\}}\right)-f(X_{s-}^{N,T})\Bigg]M^N(di,ds,dy,dz,du)\notag\\
&+\sum_{i=1}^N\int_T^t\int_{\Delta}\int_{[0,1]^N}\Bigg[f\left(X_{s-}^{N,T}-1_{\{\tau(s,i)=\oplus\}}\frac{1}{N}\sum_{j=1}^N1_{\{\tau(s,j)=\ominus,u_j\leq y+z\}}\right)-f(X_{s-}^{N,T})\Bigg]M^N(di,ds,dy,dz,du).
\end{align*}
Taking expectations in the previous identity gives for $t\geq T$
\begin{align}\label{ito_1}
	&\E\left[f(X^{N,T}_t)\right]=\E[f(X_T^{N,T})]\notag\\
	&+\E\Bigg[\int_T^t\int_{\Delta}\int_{[0,1]^N}\frac{1}{N}\sum_{i=1}^N1_{\{\tau(s,i)=\ominus\}}\Bigg[f\left(X_{s-}^{N,T}+\frac{1}{N}\sum_{j=1}^N1_{\{\tau(s,j)=\oplus,u_j\leq y\}}\right)-f(X_{s-}^{N,T})\Bigg]ds\Lambda(dy,dz)du\notag\\
	&+\E\Bigg[\int_T^t\int_{\Delta}\int_{[0,1]^N}\frac{1}{N}\sum_{i=1}^N1_{\{\tau(s,i)=\oplus\}}\Bigg[f\left(X_{s-}^{N,T}-\frac{1}{N}\sum_{j=1}^N1_{\{\tau(s,j)=\ominus,u_j\leq y+z\}}\right)-f(X_{s-}^{N,T})\Bigg]ds\Lambda(dy,dz)du\notag\\
	&=\E[f(X_T^{N,T})]+\E\Bigg[\int_T^t\int_{\Delta}X_{s-}^{N,T}\Bigg[f\left(X_{s-}^{N,T}+\frac{1}{N}\textup{Binom}(N(1-X_{s-}^{N,T}),y)\right)-f(X_{s-}^{N,T})\Bigg]ds\Lambda(dy,dz)\Bigg]\notag\\
	&+\E\Bigg[\int_T^t\int_{\Delta}\int_{[0,1]^N}(1-X_{s-}^{N,T})\Bigg[f\left(X_{s-}^{N,T}-\frac{1}{N}\textup{Binom}(NX_{s-}^{N,T},y+z)\right)-f(X_{s-}^{N,T})\Bigg]ds\Lambda(dy,dz)\Bigg].
\end{align}

Finally, by differentiating \eqref{ito_1} and taking $t=T$ we obtain \eqref{gen_ASG_1}.
\end{proof}	

\begin{proof}[Proof of Proposition \ref{prop:frequency_equal}]
The generator $\mathcal{B}^{N,T}$ has no true dependence on $T$, but only on the cardinality of the sample $S_T$ of individuals of type $\ominus$, which gives the initial frequency $|S_T|/N$ of individuals of type $\ominus$. Comparing $\mathcal{B}^{N,T}$ and $\mathcal{B}^N$ shows that they are equal, hence the processes they generate are equal in distribution. 
\end{proof}
We may thus identify the two frequency processes $(X^{N}_t)_{t\geq 0}$ and $(X_t^{N,T})_{t\geq T}$ with $X^{N}_0=|S_T|/N,$ and call it the frequency process associated with the $\Lambda$-asymmetric ancestral selection graph.

One of the big advantages of the $\Lambda$-ancestral selection graph  is that it also allows for  a $\{\mathcal{P}^T_t\}_{t\geq 0}$ measurable ancestral process,  the ancestral line counting process, which satisfies a duality relation with the frequency process. 

\begin{defn}[Ancestral process]
	Fix a sample $S_T$ of individuals at some time $T\in \mathbb{R}$.  The \emph{ancestral line counting process} $(A^{N,T}_t)_{0\leq t\leq T}$ of a sample taken at time $T$ is given by
	$$
	A^{N,T}_t=\sum_{i=1}^N 1_{\{(T-t,i)\sim (T,j), \text{ for some } (T,j)\in S_{T}\}}, \qquad t\in[0,T].	$$	
\end{defn}

Hence by definition $A_t^{N,T}$ counts the number of potential ancestors living at time $t\in[0,T]$ of individuals from the sample $S_{T}$.  

\begin{prop}
	The process $(A^{N,T}_t)_{0\leq t\leq T}$ is a $\{\mathcal{P}_t\}_{t\geq 0}$ measurable  continuous-time Markov chain with values in $[N]$ starting at $A_0^{N,T}=|S_T|$ and transition rates
\[n\mapsto \begin{cases}
	n-k &  \text{ at rate } \frac{n}{N} \int_{\Delta} \binom{n-1}{k} y^k (1-y)^{n-1-k}\Lambda(dy,dz)\notag\\&\hspace{1.3cm}+(1-\frac{n}{N})\int_{\Delta} \binom{n}{k+1} y^{k+1} (1-y)^{n-k-1}\Lambda(dy,dz),\quad k=1,...,n-1\\
	%n-k+1 &\text{ at rate } (1-\frac{n}{N})\int_{\Delta} \binom{n}{k} y^k (1-y)^{n-k}\Lambda(dy,dz),\quad k=2,...,n \\
	n+1& \text{ at rate } (1-\frac{n}{N})\int_{\Delta} [(1-y)^{n}-(1-y-z)^{n}]\Lambda(dy,dz).
\end{cases}\]
%}
\end{prop}
Since again the process doesn't depend on $T$ but only on $|S_T|,$ we drop the $T$ from the notation and call $(A^N_t)_{t\geq 0}$ the ancestral line counting process or process of potential ancestors. For probability measures (and more generally when $\|\Lambda^{\oplus}\|=\|\Lambda^{\ominus}\|$) we can apply \eqref{eq:decomp-int} and rewrite the rates in terms of the original measures:
\[n\mapsto \begin{cases}
	n-k &  \text{ at rate } \frac{n}{N} \int_{[0,1]} \binom{n-1}{k} y^k (1-y)^{n-1-k}\Lambda^{\ominus}(dy)\notag\\&\hspace{1.3cm}+(1-\frac{n}{N})\int_{[0,1]} \binom{n}{k+1} y^{k+1} (1-y)^{n-k-1}\Lambda^{\ominus}(dy),\quad k=1,...,n-1, \\
	%n-k+1 &\text{ at rate } (1-\frac{n}{N})\int_{[0,1]} \binom{n}{k} y^k (1-y)^{n-k}\Lambda^{\ominus}(dy),\quad k=2,...,n \\
	n+1& \text{ at rate } (1-\frac{n}{N})\left[\int_{[0,1]} (1-y)^{n}\Lambda^{\ominus}(dy)-\int_{[0,1]}(1-z)^{n}\Lambda^{\oplus}(dz)\right].
\end{cases}\] 
Hence we can write the dual block counting process and the duality also in terms of the original measures.  In the block counting process, the effect of the selective advantage of $\Lambda^{\oplus}$ is reduced to the question of whether there is at least one selective arrow coming from outside the sample at a time where there is no neutral arrow occurring.  However,  we would like to stress that for the construction of the ancestral selection graph and  therefore also for the \emph{pathwise} construction of the duality the existence of a monotone coupling is crucial.

\begin{proof}
Measurability is clear, as well as the initial value. In order to understand the rates,  we consider the graphical representation and go backward from sampling time $T.$ If we have currently $n$ potential ancestors, then the next event backward in time is one of the following three cases (cf. Figure \ref{fig:ancestral_trans}):
\begin{enumerate}
\item[1.] From one of the $n$ potential ancestors, $k$ neutral arrows emerge and hit $k$ of the remaining $n-1$ lines of potential ancestors. In that case, the $k$ lines coalesce with the reproducing line, reducing the number of potential ancestors by $k.$ According to the construction of the ancestral selection graph, this happens at rate $\frac{n}{N} \int_{\Delta} \binom{n-1}{k} y^k (1-y)^{n-1-k}\Lambda(dy,dz),$ leading to the first line in the above transition rates. There might be selective arrows at the same event, but these don't change the number of potential ancestors, since in that case both lines continue to be potential ancestors.
\item[2.] From one of the $N-n$ lines that don't currently belong to the set of potential ancestors there are neutral arrows to $k\geq 2$ out of the $n$ current potential ancestors.  In that case those $k$ lines merge, but the line of the reproducing individual is added to the set of potential ancestors.  This happens at rate $(1-\frac{n}{N})\int_{\Delta} \binom{n}{k} y^k (1-y)^{n-k}\Lambda(dy,dz).$
\item[3.] From one of the $N-n$ lines outside the current set of potential ancestors there is a selective arrow to at least one of the $n$ individuals from the current set of ancestors, while at the same time there are no neutral arrows to any of the $n$ current ancestors in the sample. In that case, the reproducing line is added to the set of potential ancestors, while all the other lines  remain. This happens at rate $(1-\frac{n}{N})\int_{\Delta} [(1-y)^{n}-(1-y-z)^{n}]\Lambda(dy,dz),$ since $(1-y)^{n}-(1-y-z)^{n}$  is the probability that there are no neutral arrows to individuals of the ancestal process, while there is at least one selective arrow. 
\end{enumerate}
There are no further events in the graphical construction that could change the number of potential ancestors. 
\end{proof}

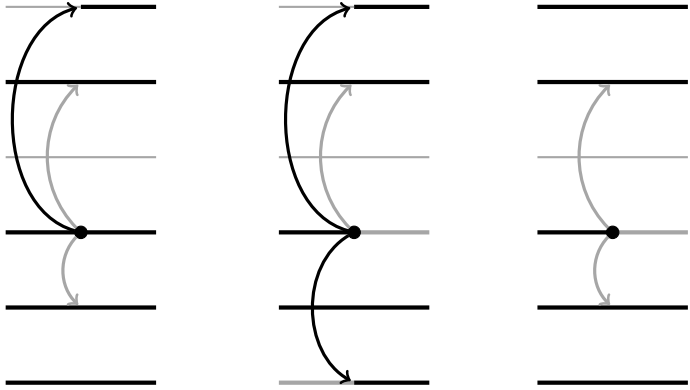
\begin{figure}[t]
        \centering
   { {\begin{minipage}[t]{100pt}
    \scalebox{1}{
        \begin{tikzpicture} %[darkstyle/.style={inner sep=0pt}]
        
        %Lines
        \draw[-,ultra thick, color=black](0,0)--(2,0);
         \draw[-, ultra thick, color=black](0,1)--(2,1);
         \draw[-, ultra thick, color=black](0,2)--(2,2);
        \draw[-,thick, color=gray!70!white](0,3)--(2,3);
          \draw[-,ultra thick, color=black](0,4)--(2,4);
         \draw[-, ultra thick, color=black](1,5)--(2,5);
        \draw[-,thick, color=gray!70!white](0,5)--(1,5);

       %Dots/Individuals
        \node (A1) at (1,0) [circle,scale=0.3]{} ;
        \node (A2) at (1,1) [circle,scale=0.3]{} ;
        \node (A3) at (1,2) [circle, scale=0.5, draw, fill=black]{} ;
        \node (A4) at (1,3) [circle, scale=0.3]{} ;
        \node (A5) at (1,4) [circle, scale=0.3]{} ;
        \node (A6) at (1,5) [circle,scale=0.3]{} ;

        %Arrows
         \draw[->, very thick] (A3) to[bend left=75] (A6);
         \draw[->, very thick, color=gray!70!white] (A3) to[bend left=45] (A5);
         \draw[->, very thick, color=gray!70!white] (A3) to[bend right=45] (A2);
         % \draw[->, very thick, color=gray!70!white] (A3) to[bend right=60] (A1);

        \end{tikzpicture}
        }   
         \end{minipage}}}
  {\begin{minipage}[t]{100pt}
   \scalebox{1}{
        \begin{tikzpicture} %[darkstyle/.style={inner sep=0pt}]
        
         %Lines
        \draw[-,ultra thick, color=black](1,0)--(2,0);
        \draw[-,ultra thick, color=gray!70!white](0,0)--(1,0);
         \draw[-, ultra thick, color=black](0,1)--(2,1);
         \draw[-, ultra thick, color=black](0,2)--(1,2);
           \draw[-, ultra thick, color=gray!70!white](1,2)--(2,2);
        \draw[-,thick, color=gray!70!white](0,3)--(2,3);
          \draw[-,ultra thick, color=black](0,4)--(2,4);
         \draw[-, ultra thick, color=black](1,5)--(2,5);
        \draw[-,thick, color=gray!70!white](0,5)--(1,5);

       %Dots/Individuals
        \node (A1) at (1,0) [circle,scale=0.3]{} ;
        \node (A2) at (1,1) [circle,scale=0.3]{} ;
        \node (A3) at (1,2) [circle, scale=0.5, draw, fill=black]{} ;
        \node (A4) at (1,3) [circle, scale=0.3]{} ;
        \node (A5) at (1,4) [circle, scale=0.3]{} ;
        \node (A6) at (1,5) [circle,scale=0.3]{} ;

        %Arrows
         \draw[->, very thick] (A3) to[bend left=75] (A6);
         \draw[->, very thick, color=gray!70!white] (A3) to[bend left=45] (A5);
        % \draw[->, very thick, color=gray!70!white] (A3) to[bend right=45] (A2);
        \draw[->, very thick, color=black] (A3) to[bend right=60] (A1);

        \end{tikzpicture}
        }   
  
    \end{minipage}}
    {\begin{minipage}[t]{100pt}
   \scalebox{1}{
        \begin{tikzpicture} %[darkstyle/.style={inner sep=0pt}]
        
         %Lines
        \draw[-,ultra thick, color=black](0,0)--(2,0);
       % \draw[-,ultra thick, color=gray!70!white](0,0)--(1,0);
         \draw[-, ultra thick, color=black](0,1)--(2,1);
         \draw[-, ultra thick, color=black](0,2)--(1,2);
           \draw[-, ultra thick, color=gray!70!white](1,2)--(2,2);
        \draw[-,thick, color=gray!70!white](0,3)--(2,3);
          \draw[-,ultra thick, color=black](0,4)--(2,4);
         \draw[-, ultra thick, color=black](0,5)--(2,5);
        %\draw[-,thick, color=gray!70!white](0,5)--(1,5);

       %Dots/Individuals
        \node (A1) at (1,0) [circle,scale=0.3]{} ;
        \node (A2) at (1,1) [circle,scale=0.3]{} ;
        \node (A3) at (1,2) [circle, scale=0.5, draw, fill=black]{} ;
        \node (A4) at (1,3) [circle, scale=0.3]{} ;
        \node (A5) at (1,4) [circle, scale=0.3]{} ;
        \node (A6) at (1,5) [circle,scale=0.3]{} ;

        %Arrows
         %\draw[->, very thick] (A3) to[bend left=75] (A6);
         \draw[->, very thick, color=gray!70!white] (A3) to[bend left=45] (A5);
         \draw[->, very thick, color=gray!70!white] (A3) to[bend right=45] (A2);
       % \draw[->, very thick, color=black] (A3) to[bend right=60] (A1);

        \end{tikzpicture}
        }   
    \end{minipage}}
      \caption{The three possible types of transitions of the ancestral process.  On the left, the reproducing individuals belongs to the current sample(thick black lines), there is one neutral (black) and two selctive (grey) arrows. The line of the individual at the tip of the neutral arrow is discarded, resp. merges with the reproducing line, while the lines hit by a selective arrow remain, as well as those lines not affected by the reproduction event. In the middle, the reproducing individuals doesn't belong to the current sample, there are two neutral and one selctive arrows. The line of the reproducing individual is thus added, the lines of the individuals at the end of a neutral arrow are discarded, resp. merge with the reproducing line. On the right, only selective arrows are present, where the incoming lines are kept, but the reproducing line is added. We therefore see a branching.}\label{fig:ancestral_trans}
\end{figure}

For $x\in S_N$, $n\in[N]$, and $t\geq 0$ we define the \emph{sampling function} $S_t(x,n)$ as follows. Let $\{u_i\}_{i\in[n]}$ be a uniformly sampled subset of $[N]$ of size $n$, and let $S_t(x,n)$ be the probability that all individuals in the sample $S_t=\{(u_i,t)\}$ are of type $\ominus$ conditional on the initial frequency of type $\ominus$ individuals at time $0$ is $x$. Note that for $t=0$
\begin{equation}
S_0(x,n)=\prod_{i=1}^n\frac{Nx+1-i}{N+1-i}=\frac{\binom{Nx}{n}}{\binom{N}{n}}.
\end{equation}
Observe that for fixed $n$ and large $N$ is of order $x^n+O(1/N),$ a fact that we will use in Proposition \ref{moment_duality} below.

Now, using the graphical representation, for any $t>0$ we can write $S_t(x,n)$ in terms of $S_0(x,n)$ and  $A^{N}_t$ started at $A_0^N=n$. Recall that an individual at time $t$ has type $\ominus$ if and only if all its potential ancestors at time $0$ have type $\ominus$. Thus the probability that $n$ individuals sampled at time $t$ is given in terms of the number of potential ancestors of these individuals, and we obtain
\begin{equation*}
	S_t(x,n)=\E[S_0(x,A^{N}_t)\,|\, A_0^N=n]=\E_n[S_0(x,A^{N}_t)].
\end{equation*}
Here, $\E$ denotes the expectation with respect to the law of the ancestral selection graph,  and $\E_n$ the expectation with respect to the Markov chain $(A_t^N)_{t\geq 0}$ started at $n.$ 
Similarly, we can write $S_t(x,n)$ in terms of $S_0(x,n)$ and  $X^{N}_t$ started at $X_0^N=x$ as 
\begin{equation*}
	S_t(x,n)=\E[S_0(X^{N}_t,n)\,|\, X_0=x]=\E_x[S_0(X^{N}_t,n)].
\end{equation*}
Hence we have shown

\begin{cor}\label{cor:duality} The processes $(X_t^N)_{t\geq 0}$ and $(A_t^N)_{t\geq 0}$ are dual with respect to the duality function $S_0,$ that is,
\[\E_x[S_0(X^{N}_t,n)]=\E_n[S_0(x,A^{N}_t)] \quad \forall t\geq 0, x\in S_N, n\in \N.\]
\end{cor}

Duality with respect to the sampling function $S_0$ is a classical concept which goes back to Cannings and Gladstien \cite{Cannings, Gladstien1, Gladstien2, Gladstien3} and was studied extensively by M\"ohle \cite{Moehle}. The same type of duality in a similar model appears in \cite{BEH}.

\section{Generalisations of the model and some remarks on the coupling}\label{sect:generalisations}

In this section we show how the construction can be extended to finite measures, which also extends the interpretation of selective advantage encoded in the stochastic order.  Further, we provide some comments concerning the question of uniqueness of the coupling. 

\medskip

\paragraph{Extension to finite measures} Assume now that $\Lambda^{\oplus}$ and $\Lambda^{\ominus}$ are finite measures on $[0,1]$ satisfying 
\begin{align}\label{eq:order}
\Lambda^{\ominus}[x,1]\leq \Lambda^{\oplus}[x,1]
\end{align}
for all $x\in[0,1].$ We thus write again by analogy to the stochastic order $\Lambda^{\ominus}\leq \Lambda^{\oplus}.$ Note that also in this case this is a  partial order on the set of finite measures.  Write $\|\Lambda^{\oplus}\|$ and $\|\Lambda^{\ominus}\|$ for their respective total masses. Now define
\begin{equation}\mu^{\oplus}:=\frac{\Lambda^{\oplus}}{\|\Lambda^{\oplus}\|}\qquad \text{and} \qquad \mu^{\ominus}:=\frac{\Lambda^{\ominus}}{\|\Lambda^{\oplus}\|}+c\delta_0,\end{equation}
where $c=1-\frac{\|\Lambda^{\ominus}\|}{\|\Lambda^{\oplus}\|}$ and $\delta_0$ denotes the Dirac measure at 0. Then clearly $\mu^{\ominus}$ and $\mu^{\oplus}$ are probability measures, and we have $\mu^{\ominus}\leq \mu^{\oplus}$ in the stochastic order if and only if $\Lambda^{\ominus}\leq \Lambda^{\oplus}$ in the sense of \eqref{eq:order}.  

We can define the $\Lambda$-asymmetric Moran model for finite measures $\Lambda^{\ominus},\Lambda^{\oplus}$ exactly as in Definition \ref{def:Lambda-Moran}. The rates of the frequency process then become 

\[
x\mapsto \begin{cases} x+\frac{k}{N}&  \text{ at rate } x\|\Lambda^{\oplus}\|\int_0^1 \binom{(1-x)N}{k} y^k (1-y)^{(1-x)N-k}\mu^{\ominus}(dy), \quad k=1,...,(1-x)N \\
x-\frac{k}{N}& \text{ at rate } (1-x) \|\Lambda^{\oplus}\|\int_0^1 \binom{xN}{k}y^k(1-y)^{xN-k}\mu^{\oplus}(dy),\quad k=1,...,xN ,
\end{cases}\] 
since an atom at 0 doesn't give a contribution.  Thus the $\Lambda$-Moran model with finite measures $\Lambda^{\oplus},\Lambda^{\ominus}$ is a simple linear time change of the Model with probability measures $\mu^{\oplus},\mu^{\ominus}$ according to the above transformation.  Reproductive events now happen at rate $\|\Lambda^{\oplus}\|.$ The ancestral selection graph can be constructed by using the selective coupling $\mu$ of $\mu^{\ominus},\mu^{\oplus}$ and then running the Poisson point process of Definition \ref{def:ancestral_selection_graph} with the coupling measure $\|\Lambda^{\oplus}\|\cdot \mu$ in the place of $\Lambda$. Hence all results in this paper also hold for finite measures provided they are ordered in the above sense. One should simply be slightly careful when applying the identity \eqref{eq:decomp-int}, since whenever $c\neq 0, $ that is, $\|\Lambda^{\ominus}\|<\|\Lambda^{\oplus}\|,$ this only holds for functions $f:[0,1]\to \R$ such that $f(0)=0.$ The measure $\Lambda$ via \eqref{eq:decomp-int} 
determines $\Lambda^{\oplus}$ and $\Lambda^{\ominus}$ only up to possible atoms at 0.
\medskip

\begin{remark}\label{rem:cases_of_reproduction} When considering finite measures, we see in particular that $\Lambda^{\ominus}\leq \Lambda^{\oplus}$ in the following two prototypical cases of selective mechanisms:
\begin{enumerate}
	\item(Faster reproduction) if $\Lambda^{\oplus}=(1+\alpha)\Lambda^{\ominus}$ for some $\alpha>0$. 
	\item(Bigger reproductive events) There exists a function $s:[0,1]\mapsto [0,1]$ such that $s(x)-x\geq 0$ and $\Lambda^{\ominus}(s(A))=\Lambda^{\oplus}(A)$.
\end{enumerate}
In the first case, in particular $\|\Lambda^{\oplus}\|=(1+\alpha)\|\Lambda^{\ominus}\|,$ hence we are in the situation of the classical ancestral selection graph by Krone and Neuhauser \cite{KN, NK} where the type with a selective advantage reproduces faster than the other. The second case is satisfied for example if $\Lambda^{\ominus}=\delta_a$ and $\Lambda^{\oplus}=\delta_b,$ with $0\leq a\leq b\leq 1$ where  $\Lambda^{\ominus}$ generates reproductive events of size $a$,  and $\Lambda^{\oplus}$ of size $b.$ 
\end{remark}
\medskip

\paragraph{Non-uniqueness of the coupling} We now turn to a discussion of the coupling from Lemma \ref{descom_measure_new}. A natural question to ask is whether the coupling $\Lambda$ we use in our construction of the ancestral selection graph is unique.  As we saw, there is always some freedom in adding atoms at 0. However, also without atoms at 0 the coupling is not unique, as can be seen from a simple example.

\begin{example}\label{ex:non-uniqueness} Consider
$\Lambda^{\ominus}(dx)=\frac{1}{2}(\delta_{1/4}(dx)+\delta_{1/2}(dx))$ and  $\Lambda^{\oplus}(dx)=\frac{1}{3}(\delta_{1/2}(dx)+\delta_{3/4}(dx)+\delta_1(dx)).$
Clearly these are probability measures on $[0,1]$ satisfying $\Lambda^{\ominus}\leq \Lambda^{\oplus}$.  Let
\begin{align*}
\Lambda_1(dy,dz)=\frac{1}{6}\sum_{i=1}^2\sum_{j=2}^4\delta_{(i/4,(j-i)/4)}(dy,dz),
\end{align*}
and
\begin{align*}
\Lambda_2(dy,dz)&=\frac{1}{4}\delta_{(1/4,1/4)}(dy,dz)+\frac{1}{12}\delta_{(1/2,0)}(dy,dz)+\frac{1}{8}(\delta_{(1/4,1/2)}(dy,dz)+\delta_{(1/4,3/4)}(dy,dz))\notag\\&+\frac{5}{24}(\delta_{(1/2,1/4)}(dy,dz)+\delta_{(1/2,1/2)}(dy,dz)).
\end{align*}

It is straightforward to check that $\Lambda_1$ and $\Lambda_2$ are two couplings of $\Lambda^{\ominus}$ and $\Lambda^{\oplus}$ in the sense of Lemma \ref{descom_measure_new}. 
%For the convenience of the reader we provide the calculations in \ref{appendix_3}.
\end{example}

\medskip

In the ancestral selection graph, we can use any coupling $\Lambda$ satisfying the conditions of Lemma \ref{descom_measure_new}. By \eqref{eq:decomp-int} this always leads to the same frequency process, and at least as long as $\|\Lambda^\oplus\|=\|\Lambda^\ominus\|$, also the rates of the ancestral line counting process remain the same.  The choice of the specific coupling is therefore not important on the level of these processes. However, the fact that at least one selective coupling exists is of course crucial in the construction of the ASG, and therefore pathwise construction of the duality.  It is also important in Section \ref{sect:Griffiths}.

In spite of this non-uniqueness, Strassen's theorem as used in the proof of Lemma \ref{descom_measure_new} provides us with a characterisation of the coupling $\rho$ introduced after Lemma \ref{descom_measure_new}. In fact, as we will argue, it minimizes the variance of the number of potential ancestors in the construction of the $\Lambda$-asymmetric ancestral graph given in Definition \ref{def:ancestral_selection_graph} in the case in which $\Lambda^\ominus$ and $\Lambda^\oplus$ are  both probability measures. To be more precise, denote for $n\in\mathbb{N}$.
	\[
	c(y,z):=(z-y)^2, \qquad y,z\in[0,1]^2,
	\]
	and consider the following optimal transport problem, consisting  in finding a probability measure $\gamma^*$ on $[0,1]^2$ such that the following infimum is achieved
	\begin{align}\label{opt_prob}
		V(n,\Lambda^\ominus,\Lambda^\oplus):=\inf\left\{\int_{[0,1]^2}c(y,z)\gamma(dy,dz): \gamma\in \Gamma(\Lambda^\ominus,\Lambda^\oplus)\right\},
	\end{align}
	where $\Gamma(\Lambda^\ominus,\Lambda^\oplus)$ is the set of probability measures on $[0,1]^2$ with marginals $\Lambda^\ominus$, $\Lambda^\oplus$ on $[0,1]$. We first note that by construction
	\begin{align}\label{opt_tran_2}
		\int_{\Delta}z^2\Lambda(dy,dz)&=\int_{[0,1]^2}\left(z-y\right)^2\rho(dy,dz)=\E\left[c(Z^\ominus,Z^\oplus)\right],
	\end{align}
	where $Z^\ominus,Z^\oplus$, given in the proof of Lemma \ref{descom_measure_new}, satisfy that $(Z^\ominus,Z^\oplus)=(F_\ominus^{-1}(U),F_\oplus^{-1}(U))$.
	
	On the other hand, the function $c$ is symmetric and satisfies the `Monge" conditions, i.e. $c$ is continuous and satisfies that for $y'\geq y$ and $z'\geq z$
	\begin{align*}
		c(y',z')-c(y,z')-c(y',z)+c(y,z)\leq 0.
	\end{align*}
	Hence, as in Section 7.1 in \cite{RK} we have that the solution to the optimization problem given in \eqref{opt_prob} is given by a measure $\gamma^*$ with distribution $F^*(y,z):=\min\{F_\ominus(y),F_\oplus(z)\}$ for $y,z\in[0,1]^2$, or in terms of random variables, $\gamma^*$ is the distribution function of the random vector $(F_\ominus^{-1}(U),F_\oplus^{-1}(U))$, 
	where $U$ a uniform random variable. %In other words
Comparing with \eqref{opt_tran_2} we see that $\gamma^*=\rho$ and therefore the measure $\rho$ from our coupling is the solution to the minimisation problem \eqref{opt_prob} i.e.
	\[
	V(n,\mu_1,\mu_2)=\int_{[0,1]^2}\left(z-y\right)^2\rho(dy,dz),
	\]
	for each $n\in\mathbb{N}$. 
To arrive at the interpretation in terms of selective arrows, first observe that by definition of the ancestral selection graph, the number of selective arrows for any coupling $\Lambda$ is given by
	\[\int_{\Delta}z\Lambda(dy,dz)=\int_{[0,1]^2}\left(z-y\right)\rho(dy,dz), \]
	which is a constant, since by \eqref{eq:decomp-int}
	\[\int_{\Delta}z\Lambda(dy,dz)=\int_{[0,1]}(y+z)\Lambda^{\oplus}(dz)-\int_{[0,1]}y\Lambda^{\ominus}(dy):=C.\]
Therefore the expected number of selective arrows associated with any coupling in the class $\Gamma(\Lambda^{\ominus},\Lambda^{\oplus})$ is constant and equal to $C$. On the other hand,  the variance of the number of selective arrows is given by 
\[\int_{\Delta}z^2\Lambda(dy,dz)-C^2=\int_{[0,1]^2}\left(z-y\right)^2\rho(dy,dz)-C^2.\]
Hence minimising $\int_{[0,1]^2}\left(z-y\right)^2\gamma(dy,dz)$ in the class $\Gamma(\Lambda^{\ominus},\Lambda^{\oplus})$ also minimises this variance.

\section{Scaling Limits}\label{scaling_limits}
In this section we will study the scaling limit of the frequency process associated to the $\Lambda$-asymmetric ancestral selection graph introduced in Section \ref{sect:adaptation-selection}. We start by showing that the limiting object is well defined, which is the content of the next result.  We defer the proof to \ref{appendix_1}.
\begin{prop}\label{sde_exis}
	Assume that $\Lambda$ is a measure on $\Delta=\{(y,z)\in[0,1]^2:y+z\in[0,1]\}$ such that
	\begin{align}\label{int}
		\int_{\Delta}(y^2+z)\Lambda(dy,dz)<\infty.
	\end{align}
	Then, for any $ x\in[0,1]$ there exists a unique strong solution $Y=(Y_t)_{t\geq0}$ to the following stochastic differential equation
	\begin{align}\label{sde_n_2}
		Y_t&=x+\int_0^t\int_0^1\int_{\Delta}\left(y(1-Y_{s-})1_{\{u\leq Y_{s-}\}}-(y+z)Y_{s-}1_{\{u\geq Y_{s-}\}}\right)1_{\{Y_{s-}\in[0,1]\}}\tilde{N}(ds,du,dy,dz)\notag\\
		&-\int_{\Delta}z \Lambda(dy,dz)\int_0^t(1-Y_{s-})Y_{s-}1_{\{Y_{s-}\in[0,1]\}}ds,
	\end{align} 
	where $ N$ is a Poisson random measure on $(0,\infty)\times[0,1]\times\Delta$ with intensity measure $dt\times du\times\Lambda(dy,dz)$ and $\tilde{N}(ds,du,dy,dz):=N(ds,du,dy,dz)-dsdu\Lambda(dy,dz)$ denotes the compensated random measure associated to $N$. We refer to the process $Y$ as the \emph{$\Lambda$-asymmetric frequency process}.
\end{prop}
\begin{rem}
	Note that if the stronger condition 
	%\begin{equation*}%\label{cond_1}
		$\int_{\Delta}(y+z)\Lambda(dy,dz)<\infty$,
	%\end{equation*}
	holds, then the process $Y$, solution to the SDE given in \eqref{sde_n_2}, can be described as the solution to the simpler SDE given by
	\begin{align*}%\label{sde_n_1}
		Y_t=x+\int_0^t\int_{\Delta}\int_0^1\left(y(1-Y_{s-})1_{\{u\leq Y_{s-}\}}-(y+z)Y_{s-}1_{\{u\geq Y_{s-}\}}\right)1_{\{Y_{s-}\in[0,1]\}}N(ds,du,dy,dz).
	\end{align*}
\end{rem}

In the following result we derive the convergence to the process $Y$, of the frequency processes associated to the $\Lambda$-ancestral selection graph introduced in Section \ref{sect:adaptation-selection}.  In doing so we need to take some care in dealing with the possible singularity at 0 of the measure $\Lambda$. To this end, we will apply a truncation procedure at 0, and study a sequence of frequency processes $Y^N$ that correspond to the truncated $\Lambda$-measures. 
\medskip

Fix a measure $\Lambda$ on $\Delta$ such that $\int_{\Delta}(y^2+z)\Lambda(dy,dz)<\infty$, an initial frequency $x\in[0,1]$ and a number $\alpha\in(0,1/2)$.
% and $\mu^{(i)}\in \mathcal{M}(\Delta)$. 
For each $N\in\mathbb{N}$ we define the truncated frequency process $Y^N=(Y_t^N)_{t\geq 0}$ with initial value  $Y^N_0=\lfloor xN \rfloor/N$ as the frequency process of an asymmetric $\Lambda$-Moran model with reproduction mechanism 
$\Lambda^{(N)}(A):=\int_{\Delta^N}1_A(y,z)\Lambda(dy,dz)$, where 
$\Delta^N:=\{(y,z)\in\Delta:y^2>1/N^{\alpha}\}$.

Note that the measure $\Lambda^{(N)}$ is a finite measure on $\Delta$, indeed by the definition of $\Delta^N$ we have
\[
\Lambda^{(N)}(\Delta)=\Lambda(\Delta^N)\leq N^{\alpha}\int_{\Delta}(y^2+z)\Lambda(dy,dz)<\infty.
\]
Hence $Y^N$ is well defined. We will now show the convergence of the sequence of process $(Y^N)_{N\in\mathbb{N}}$ as the total size of the population $N$ grows to infinity to the  $\Lambda$-asymmetric frequency process $Y$ defined in Proposition \ref{sde_exis}.  The proof is deferred to \ref{appendix_2}.
\begin{prop}\label{asy_a}
	The sequence $(Y^N)_{N\in\mathbb{N}}$ converges weakly in $\mathbb{D}([0,T],[0,1])$ to a limit $Y$, where $Y$ is the $\Lambda$-asymmetric frequency process given in \eqref{sde_n_2}.
\end{prop}
Let $A^{N}:=(A^N_t)_{t>0}$ denote the ancestral process associated to the frequency process $Y^N$ defined in Proposition \ref{asy_a}. In the next result we show the convergence of the sequence of processes $(A^N)_{N\in\mathbb{N}}$ as the total size of the population $N$ grows to infinity.
\begin{prop}\label{anc_proc}
The sequence $(A^{N})_{N\in\N}$ converges weakly in $\mathbb{D}([0,T],\mathbb{N})$ to a limit $A=(A_t)_{t\geq0}$, where $A$ is a continuous-time Markov chain with values in $\mathbb{N}$, whose generator has the following transition rates
\[m\mapsto \begin{cases}
	m-k+1 &\text{ at rate } \int_{\Delta} \binom{m}{k} y^k (1-y)^{m-k}\Lambda(dy,dz),\quad k=2,...,m \\
	m+1& \text{ at rate } \int_{\Delta} [(1-y)^{m}-(1-y-z)^{m}]\Lambda(dy,dz).
\end{cases}\]
\end{prop}
\begin{proof}
Fix $m\in \mathbb{N}$. Using that $|(1-y)^m-(1-y-z)^m|\leq m z$, for $y,z\in\Delta$ we obtain that \eqref{int} together with dominated convergence gives
\begin{align*}
\lim_{N\to\infty}\left(1-\frac{m}{N}\right)\int_{\Delta^{N}}\left[(1-y)^m-(1-y-z)^m\right]\Lambda(dy,dz)=\int_{\Delta}\left[(1-y)^m-(1-y-z)^m\right]\Lambda(dy,dz).
\end{align*}
On the other hand, for $k\geq 2$, by \eqref{int} and dominated convergence
\begin{align*}
	\lim_{N\to\infty}\left(1-\frac{m}{N}\right)\int_{\Delta^{N}}\binom{m}{k}y^k(1-y)^{m-k}\Lambda(dy,dz)=\int_{\Delta}\binom{m}{k}y^k(1-y)^{m-k}\Lambda(dy,dz).
\end{align*}
By the definiton of $\Delta^N$ we note that dominated convergence gives
\begin{align*}
\lim_{N\to\infty}\frac{m}{N}\int_{\Delta^{N}}my(1-y)^{m-1}\Lambda(dy,dz)\leq\lim_{N\to\infty} m^2\frac{1}{N^{1-\alpha/2}}\int_{\Delta^{N}}y^2\Lambda(dy,dz)=0.
\end{align*}
Similarly, for $k\geq 2$, proceeding as in the previous identities we get
\begin{align*}
\lim_{N\to\infty}\frac{m}{N}\int_{\Delta^{N}}\binom{m}{k}y^k(1-y)^{m-k}\Lambda(dy,dz)=0.
\end{align*}
Therefore, by Problem 3(ii) in \cite{EK} we obtain the result.
\end{proof}
As expected we show in the next result that the processes $Y$ and $A$ are moment duals. To this end, we note that the infinitesimal generator $\mathcal{A}$ of the process $A$ is given for any $f:\mathbb{N}\mapsto\mathbb{N}$ by
\begin{align}\label{inf_2}
\mathcal{A}f(n)&=\sum_{k=1}^{n}(f(n-k+1)-f(n))\int_{\Delta} \binom{n}{k} y^k (1-y)^{n-k}\Lambda(dy,dz)\notag\\&+(f(n+1)-f(n))\int_{\Delta} [(1-y)^{n}-(1-y-z)^{n}]\Lambda(dy,dz).
\end{align}
\begin{prop}\label{moment_duality}
Consider the $\Lambda$-asymmetric frequency process $Y$ defined in Proposition \ref{sde_exis} and the ancestral process $A$ with rates given in Proposition \ref{anc_proc}. Then the processes $Y$ and $A$ are moment duals, i.e. for any $x\in[0,1]$ and $n\in\mathbb{N}$
\[
\mathbb{E}_x\left[Y_t^n\right]=\E_n\left[x^{A_t}\right], \qquad  \text{for all $t\geq0$}.
\]
\end{prop}
\begin{proof}
We will consider $\N$ endowed with the discrete topology and $\N\times [0,1]$ with the product topology. We denote for every fixed $x\in[0,1]$, $H(n,x)=x^n$, which is bounded and continuous. %Note that $\N_0\cup \{\Delta\}$ is equipped with the discrete topology and $\N_0\cup \{\Delta\}\times [0,1]$ with the product topology. 
In addition, for every fixed $k\in \N$, $H(k,x)=x^k$ is continuous. Therefore, we conclude that $H:\N\times [0,1]\mapsto [0,1]$ is continuous.

We observe that $H(\cdot,n)$ is a polynomial in $[0,1]$ for fixed $n\in\mathbb{N}$. This fact clearly implies that $H(\cdot,n)\in\mathcal{C}^2([0,1])$ and hence it lies in the domain of the generator $\mathcal{B}$ of the $\Lambda$-asymmetric frequency process $Y$ as in \eqref{inf_1}. Therefore, the process
\[
H(n,Y_t)-\int_0^t\mathcal{B}H(n,Y_s)ds,
\]
is a martingale.

Additionally, we have that for fixed $x\in[0,1]$ the function $H(\cdot,x)$ lies in the domain of the generator $\mathcal{A}$, which implies that the process
\[
H(A_t,x)-\int_0^t\mathcal{A}H(A_s,x)ds,
\]
is also a martingale.

We will compute $\mathcal{B}H(n,x)$ for $x\in[0,1]$ and $n\in\mathbb{N}$, then using \eqref{inf_1}
\begin{align}\label{dual_1}
\mathcal{B}H(n,x)=\int_{\Delta}\left[x(x+y(1-x))^n+(1-x)(x-(y+z)x)^n-x^n\right]\Lambda(dy,dz), \ x\in[0,1], n\in\mathbb{N}.
\end{align}
We note that
\begin{align*}
x(x+y(1-x))^n=x\sum_{k=0}^n\binom{n}{k}y^kx^{n-k}(1-y)^{n-k}=x^{n+1}(1-y)^n+\sum_{k=1}^n\binom{n}{k}y^kx^{n-k+1}(1-y)^{n-k}.
\end{align*}
Additionaly,
\begin{align*}
(1-x)(x-(y+z)x)^n=(x^{n}-x^{n+1})(1-y-z)^{n},
\end{align*}
and finally
\begin{align*}
x^{n}=x^{n}(1-y)^n+x^n\sum_{k=1}^n\binom{n}{k}y^k(1-y)^{n-k}.
\end{align*}
Hence, using the previous identities gives
\begin{align}\label{dual_2}
x(x+y(1-x))^n+(1-x)(x-(y+z)x)^n-x^n&=\sum_{k=2}^n\binom{n}{k}y^k(1-y)^{n-k}(x^{n-k+1}-x^n)\notag\\
&+(x^{n+1}-x^n)\left[(1-y)^n-(1-y-z)^{n}\right].
\end{align}
Using \eqref{dual_2} in \eqref{dual_1} we obtain
\begin{align*}
\mathcal{B}H(n,x)&=\sum_{k=2}^n(x^{n-k+1}-x^n)\int_{\Delta}\binom{n}{k}y^k(1-y)^{n-k}\Lambda(dy,dz)\notag\\
&+(x^{n+1}-x^n)\int_{\Delta}\left[(1-y)^n-(1-y-z)^{n}\right]\Lambda(dy,dz)=\mathcal{A}H(n,x).
\end{align*}
Finally an application of Proposition 5 in \cite{CGP} gives the result. 
\end{proof}

\section{Griffiths representation and the probability of fixation}\label{sect:Griffiths}
We motivate our next result by recalling a result for the two-type $\Lambda$-Fleming Viot process, which is defined as the solution to the SDE
	\begin{align*}%\label{FV}
		Z_t&=y+\int_0^t\int_0^1\int_{\Delta}\left(y(1-Z_{s-})1_{\{u\leq Z_{s-}\}}-yZ_{s-}1_{\{u\geq Z_{s-}\}}\right)1_{\{Z_{s-}\in[0,1]\}}\tilde{N}(ds,du,dy)\notag\\&-\int_0^ts(1-Z_{s-})Z_{s-}1_{\{Z_{s-}\in[0,1]\}}ds,\quad t\geq0.
	\end{align*}
	Independently and using different techniques, Foucart \cite{Foucart} and Griffiths \cite{Griff} obtained explicit conditions for the two-type $\Lambda$-Fleming Viot process to be absorbed at 0 almost surely, or to be absorbed in either of the boundaries $\{0,1\}$ with positive probability. In particular, they observed that for every $\Lambda$ it is always possible to find a small enough selection parameter $s$ such that $\P(\lim_{t\rightarrow \infty}Z_t=1)\P(\lim_{t\rightarrow \infty}Z_t=0)>0$. Foucart uses a duality technique which relies on the observation that the $\Lambda$-Fleming Viot process can be absorbed at $1$ if and only if its dual, that we denote $(F_t)_{t\geq 0}$, is positive recurrent. The process $(F_t)_{t\geq 0}$ coalesces like the block counting process of the $\Lambda$-coalescent and branches at rate $ns$.  In our model, the dual process is given by the ancestral line counting process $(A_t)_{t\geq 0}$. It coalesces at the same rate and branches at rate $ \int_{0}^{1}\int_{0}^{1}\left[(1-y)^{n-1}-(1-y-z)^{n-1}\right]\Lambda(dy,dz)$ which is smaller than $ns$ for all $n>n(s)$, for some $n(s)\gg 1$, which exists for every $s$. This implies that $(F_t)_{t\geq 0}$ reflected in $n(s)$ stochastically dominates $(A_t)_{t\geq 0}$. In turn this implies that $(A_t)_{t\geq 0}$ is positive recurrent. A duality argument allows to conclude that the limiting frequency process $(Y_t)_{t\geq 0}$ in equation \eqref{sde_exis} fullfils $\P(\lim_{t\rightarrow \infty}Y_t=1)\P(\lim_{t\rightarrow \infty}Y_t=0)>0$.

Here, instead of formalising the previous argument, we will exploit Griffiths' technique in order to compute a semi-explicit expression for the probability of fixation of type $\ominus$ individuals in a  $\Lambda$-asymmetric frequency process. To this end, we will start by obtaining an alternative expression for its generator $\mathcal{B}$, given in \eqref{inf_1}, following the ideas of Theorem 1 in \cite{Griff}. This is given in the next result.
We define the measure $\tilde{\Lambda}$ on the simplex $\Delta$ by
$
\tilde{\Lambda}(dy,dz):=(y^2+z)\Lambda(dy,dz).
$
\begin{prop}
Let $\mathcal{B}$ be the infinitesimal generator of a $\Lambda$-asymmetric frequency process as in \eqref{inf_1}. Let $U,V,Y,Z$ be independent random variables, such that $V$ has a uniform distribution on $[0,1]$, $U$ has density $2u$ with respect to the Lebesgue measure, and $(Y,Z)$ is distributed according to $\tilde{\Lambda}/\|\tilde{\Lambda}\|$. Then, for any $f\in\mathcal{C}^{2}([0,1])$,
\begin{align*}
	\mathcal{B}f(x)=\frac{\|\tilde{\Lambda}\|}{2}x(1-x)\E\left[f''(x(1-UY)+UYV)\frac{Y^2}{Z+Y^2}-2 f'(x-xY-xZV)\frac{Z}{Z+Y^2}\right].
\end{align*}
\end{prop}
\begin{proof}
Using that $V$ has a uniform distribution on $[0,1]$ we obtain that
\begin{align}\label{griff_1}
\E\left[f''(x(1-UY)+UYV)\frac{Y^2}{Z+Y^2}\right]&=\E\left[\frac{Y^2}{Z+Y^2}\int_0^1f''(x(1-UY)+UYv)dv\right]\notag\\
&=\E\left[\frac{Y}{Z+Y^2}\left[f'(x(1-UY)+UY)-f'(x(1-UY))\right]\frac{1}{U}\right]\notag\\
&=\E\left[\frac{2}{Z+Y^2}\left(\frac{f(x(1-Y)+Y)-f(x)}{1-x}+\frac{f(x(1-Y))-f(x)}{x}\right)\right]\notag\\
&=\frac{2}{\|\tilde{\Lambda}\|}\int_{\Delta}\left[\frac{f(x(1-y)+y)-f(x)}{1-x}+\frac{f(x(1-y))-f(x)}{x}\right]\frac{\tilde{\Lambda}(dy,dz)}{z+y^2}
\end{align}
where in the third equality we used that $U$ has density given by $2u$ on $[0,1]$.

In a similar way, we obtain
\begin{align}\label{griff_2}
\E\left[\frac{Z}{Z+Y^2}f'(x-xY-xZV)\right]&=-\E\left[\frac{Z}{Z+Y^2}\left(f(x-xZ-xY)-f(x-xY)\right)\frac{1}{xZ}\right]\notag\\
&=-\frac{1}{\|\tilde{\Lambda}\|}\int_{\Delta}\left[\frac{f(x-xy-xz)-f(x-xy)}{x}\right]\frac{\tilde{\Lambda}(dy,dz)}{z+y^2}.
\end{align}
Therefore, using \eqref{griff_1} and \eqref{griff_2} we obtain for $x\in[0,1]$
\begin{align*}
\frac{\|\tilde{\Lambda}\|}{2}&x(1-x)\E\left[f''(x(1-UY)+UYV)\frac{Y^2}{Z+Y^2}-2 f'(x-xZV)\frac{Z}{Z+Y^2}\right]\notag\\&=\int_{\Delta}\left[x(f(x(1-y)+y)-f(x))+(1-x)(f(x(1-y))-f(x))\right]\frac{\tilde{\Lambda}(dy,dz)}{z+y^2}\notag\\
&+\int_{\Delta}(1-x)\left[f(x(1-y-z))-f(x(1-y))\right]\frac{\tilde{\Lambda}(dy,dz)}{z+y^2}\notag\\
&=\int_{\Delta}\left[x(f(x(1-y)+y)-f(x))+(1-x)(f(x(1-y-z))-f(x))\right]\frac{\tilde{\Lambda}(dy,dz)}{z+y^2}\notag\\
&=\mathcal{B}f(x).
\end{align*}
\end{proof}

Denote by $p(x)$ for $x\in[0,1]$ the probability of fixation of type $\ominus$ in the $\Lambda$-asymmetric frequency process,  where $x$ denotes the initial frequency of type $\ominus$ individuals.  Since $p$ is a harmonic function for the generator $\mathcal{B}$, we see from the previous proposition that it satisfies for $x\in(0,1)$
\begin{align}\label{eq_prob}
	\E\left[p''(x(1-UY)+UYV)\frac{Y^2}{Z+Y^2}-2p'(x-xY-xZV)\frac{Z}{Z+Y^2}\right]=0,
\end{align}
with boundary conditions given by $p(0)=0$ and $p(1)=1$.

We note that
\begin{align}\label{griff_0}
	\E\left[p''(x(1-UY)+UYV)\frac{Y^2}{Z+Y^2}\right]=\E\left[\frac{p'(x(1-W)+W)-p'(x(1-W))}{W}\frac{Y^2}{Z+Y^2}\right],
\end{align}
with $W:=UY$.

In order to find the solution to \eqref{eq_prob}, following \cite{Griff}, we will consider polynomials of the form
\begin{align}\label{poly}
	h_n(x)=\sum_{r=0}^{n}a_{n,r}x^r,
\end{align}
with $h_0(x)=1$.

Then, we will prove that we can take a choice of coefficients $\{a_{n,r}\}_{0\leq r\leq n}$ such that
\begin{align}\label{cond_griff}
	\E\left[\frac{h_n(x(1-W)+W)-h_n(x(1-W))}{W}\frac{Y^2}{Z+Y^2}\right]=n\E\left[h_{n-1}\left(x(1-Y-ZV)\right)\frac{Z}{Z+Y^2}\right].
\end{align}
with the choice
\begin{equation}\label{first}
a_{n,n}=\prod_{i=1}^{n-1}\frac{\E\left[(1-Y-ZV)^i\displaystyle\frac{Z}{Z+Y^2}\right]}{\E\left[(1-W)^i\displaystyle\frac{Y^2}{Z+Y^2}\right]}, \qquad n=0,1\dots
\end{equation}
for the diagonal elements. Indeed, by \eqref{poly}
\begin{align*}
	\E\left[\frac{h_n(x(1-W)+W)-h_n(x(1-W))}{W}\frac{Y^2}{Z+Y^2}\right]&=\E\left[\frac{Y^2}{Z+Y^2}\sum_{r=0}^n\frac{a_{n,r}}{W}\left[(x(1-W)+W)^r-(x(1-W))^r\right]\right]\notag\\
	&=\sum_{r=1}^na_{n,r}\sum_{j=0}^{r-1}\binom{r}{j}\E\left[(1-W)^jW^{r-j-1}\frac{Y^2}{Z+Y^2}\right]x^{j}\notag\\
	&=\sum_{j=0}^{n-1}\sum_{r=j+1}^n\binom{r}{j}\E\left[(1-W)^jW^{r-j-1}\frac{Y^2}{Z+Y^2}\right]a_{n,r}x^j.
\end{align*}
Hence in order for \eqref{cond_griff} to hold we need that
\begin{align*}
	\sum_{j=0}^{n-1}\sum_{r=j+1}^n\binom{r}{j}\E\left[(1-W)^jW^{r-j-1}\frac{Y^2}{Z+Y^2}\right]a_{n,r}x^j=\sum_{j=0}^{n-1}na_{n-1,j}\E\left[(1-Y-ZV)^j\frac{Z}{Z+Y^2}\right]x^j,
\end{align*}
or equivalently that
\begin{align}\label{recursion}
	a_{n-1,j}=\sum_{r=j+1}^n\binom{r}{j}\frac{\E\left[(1-W)^jW^{r-j-1}\displaystyle\frac{Y^2}{Z+Y^2}\right]}{n\E\left[(1-Y-ZV)^j\displaystyle\frac{Z}{Z+Y^2}\right]}a_{n,r}.
\end{align}
Following the discussion in \cite{Griff}, we can determine the coefficients $\{a_{n,j}\}_{j=0}^{n-1}$ of the polynomial $h_{n}$, from the coefficients $\{a_{n-1,j}\}_{j=0}^{n}$ of the polynomial $h_{n-1}$. Indeed, we can take $a_{n,n}$ as in \eqref{first} and then, using \eqref{recursion}, recursively obtain $a_{n,j}$ for $j=n-1,\dots,0$. We notice that the coefficient $a_{n,0}$ can be chosen arbitrarly, and will be specified later in the construction of the probability of fixation $p$.

Now let us return to the probability of fixation $p$, where for the first derivative $p'(x)$ we make the ansatz 
\begin{align}\label{der}
	p'(x)= A\sum_{n=1}^{\infty}2^nc_nh_{n-1}(x),
\end{align}
with $p(1)=1$ and $p(0)=0$. Using  \eqref{der} in \eqref{eq_prob}, together with \eqref{griff_0} and \eqref{cond_griff} gives
\begin{align*}
	&\E\left[\frac{p'(x(1-W)+W)-p'(x(1-W))}{W}\frac{Y^2}{Z+Y^2}-2 p'(x(1-Y-ZV))\frac{Z}{Z+Y^2}\right]\notag\\
	&=A\sum_{n=1}^{\infty}2^nc_n(n-1)\E\left[h_{n-2}\left(x(1-Y-ZV)\right)\frac{Z}{Z+Y^2}\right]-2A\sum_{n=1}^{\infty}2^nc_n\E\left[h_{n-1}\left(x(1-Y-ZV)\right)\frac{Z}{Z+Y^2}\right].
\end{align*}
By chosing $
c_n=\frac{1}{(n-1)!},
$ for $n\in\mathbb{N}$, we obtain
\begin{align*}
	&\E\left[\frac{p'(x(1-W)+W)-p'(x(1-W))}{W}\frac{Y^2}{Z+Y^2}-2 p'(x(1-Y-ZV))\frac{Z}{Z+Y^2}\right]\notag\\
	&=A\sum_{n=2}^{\infty}\frac{2^n}{(n-2)!}\E\left[h_{n-2}\left(x(1-Y-ZV)\right)\frac{Z}{Z+Y^2}\right]-A\sum_{n=1}^{\infty}\frac{2^{n+1}}{(n-1)!}\E\left[h_{n-1}\left(x(1-Y-ZV)\right)\frac{Z}{Z+Y^2}\right]=0.
\end{align*}
Hence, $p$ is a solution to \eqref{eq_prob}. Now, by integrating \eqref{der} we obtain that
\begin{align*}
	p(x)=A\sum_{n=1}^{\infty}\int_0^x \frac{2^n}{(n-1)!}h_{n-1}(u)du.
\end{align*}
So if we choose $\{h_n(0)\}_{n\geq1}$ such that
$
\int_0^1 nh_{n-1}(u) du=1,
$
then by the fact that $p(1)=1$ and $p(0)=0$ we obtain that
\begin{align*}
	1=p(1)-p(0)=A\sum_{n=1}^{\infty}\frac{2^n}{n!}\int_0^1nh_{n-1}(u)du=A(e^{2}-1),
\end{align*}
and hence $A=(e^{2}-1)^{-1}$.

So putting the pieces together we have that
\[
p(x)=(e^{2}-1)^{-1}\sum_{n=1}^{\infty}\frac{2^n}{n!}H_n(x),
\]
where $H_n(x):=\int_0^x nh_{n-1}(u)du$, and $\{h_n\}_{n\geq 1}$ satisfies \eqref{cond_griff}.
\bigskip

The previous discussion leads to the following main result of this section. 
\begin{prop}
	The fixation probability of type $\ominus$ individuals is given by
	\begin{equation*}
		p(x)=(e^{2}-1)^{-1}\sum_{n=1}^{\infty}\frac{2^n}{n!}H_n(x),\qquad x\in[0,1],
	\end{equation*}
	where the polynomials $\{H_n\}_{n=0}^{\infty}$ are given by
	\[
	H_n(x)=\int_0^xnh_{n-1}(u)du,\qquad x\in[0,1],
	\]
	with $\{h_n\}_{n=0}^{\infty}$ given in \eqref{poly}, \eqref{first}, and \eqref{recursion} and the constants $\{h_n(0)\}_{n=0}^{\infty}$ chosen so that
	$
	\int_0^1nh_{n-1}=1.
	$
\end{prop}

\begin{appendix}

\section{Proof of Proposition \ref{sde_exis}}\label{appendix_1}
%\begin{proof}
	Let us denote for $(x,y,z,u)\in[0,1]\times\Delta\times[0,1]$
	\[
	g(x,y,z,u):=y(1-x)1_{\{u<x\}}-(y+z)x1_{\{u\geq x\}}.
	\]
	For fixed $(y,z,u)\in\Delta\times[0,1]$ we have
	\begin{equation}\label{eu_1}
		x+g(x,y,z,u)= \begin{cases} y+x(1-y) &\mbox{if } u<x, \\
			x(1-(y+z)) & \mbox{if }u\geq x. \end{cases}
	\end{equation}
	First, we note that for fixed $(y,z,u)\in\Delta\times[0,1]$
	\begin{align*}
		0\leq x+g(x,y,z,u)=x+y(1-x)1_{\{u<x\}}-(y+z)x1_{\{u\geq x\}}\leq 1.
	\end{align*}
	Hence, by a modification of Proposition 2.1 in \cite{FL} (see also Corollary 6.2 in \cite{LP}) we obtain that $\mathbb{P}\left(Y_t\in[0,1] \ \text{for all $t\geq0$}\right)=1$.
	
	On the other hand, \eqref{eu_1} implies that for fixed $(y,z,u)\in\Delta\times[0,1]$ the mapping $x\mapsto x+g(x,y,z,u)$ is non-decreasing for $x\in[0,1]$.
	
	Denote for $x\in[0,1]$
	\begin{equation*}
		b(x):=\left(\int_{\Delta}z\Lambda(dy,dz)\right)x(1-x).
	\end{equation*}
	Then, for $x_1,x_2\in[0,1]$
	\begin{align}\label{eu_2}
		|b(x_1)-b(x_2)|=|x_2(1-x_2)-x_1(1-x_1)|\int_{\Delta}z\Lambda(dy,dz)\leq 2\int_{\Delta}z\Lambda(dy,dz)|x_2-x_1|.
	\end{align}
	Similarly, for $x_1,x_2\in[0,1]$
	\begin{align}\label{lips_g}
		|g(x_1,y,z,u)-g(x_2,y,z,u)|&=y(x_2-x_1)1_{\{u\leq x_1\wedge x_2\}}-(y(1-x_2)+(y+z)x_1)1_{\{x_1\leq u\leq x_2\}}\notag\\
		&+(y(1-x_1)+(y+z)x_2)1_{\{x_2\leq u\leq x_1\}}-(y+z)(x_1-x_2)1_{\{x_1\vee x_2\leq u\}}.
	\end{align}
	Using \eqref{lips_g} we obtain for $x_1,x_2\in[0,1]$
	\begin{align}\label{eu_3}
		\int_{\Delta}\int_0^1&|g(x_1,y,z,u)-g(x_2,y,z,u)|^2du\Lambda(dy,dz)\notag\\
		&\leq\int_{\Delta}\int_0^1\Big[y^2|x_2-x_1|^21_{\{u\leq x_1\wedge x_2\}}+(y(1-x_2)+(y+z)x_1)^21_{\{x_1\leq u\leq x_2\}}\notag\\
		&+(y(1-x_1)+(y+z)x_2)^21_{\{x_2\leq u\leq x_1\}}+(y+z)^2|x_1-x_2|^21_{\{x_1\vee x_2\leq u\}}\Big]du \Lambda(dy,dz)\notag\\
		&\leq 6|x_2-x_1|\int_{\Delta}(y^2+z)\Lambda(dy,dz).
	\end{align}
	Finally we note that for $x\in[0,1]$
	\begin{align}\label{eu_4}
		|b(x)|^2=\left(\int_{\Delta}z\Lambda(dy,dz)\right)^2x^2(1-x)^2\leq \left(\int_{\Delta}z\Lambda(dy,dz)\right)^2x^2,
	\end{align}
	and
	\begin{align}\label{eu_5}
		\int_{\Delta}\int_0^1|g(x,y,z,u)|^2du\Lambda(dy,dz)&=\int_{\Delta}\int_0^1\left[y^2(1-x)^21_{\{u< x\}}+(y+z)^2x^21_{\{u\geq x\}}\right]du\Lambda(dy,dz)\notag\\
		&\leq \int_{\Delta}\left[y^2(1-x)^2x+x^2(y+z)^2(1-x)\right]\Lambda(dy,dz)\notag\\
		&\leq (1+x^2)\int_{\Delta}(y^2+z)\Lambda(dy,dz).
	\end{align}
	Therefore, there exists $K>0$ such that for $x\in[0,1]$
	\begin{align}\label{eu_6}
		|b(x)|^2+\int_{\Delta}\int_0^1|g(x,y,z,u)|^2du\Lambda(dy,dz)\leq K(1+|x|^2).
	\end{align}
	Hence using \eqref{eu_2}, \eqref{eu_3} and \eqref{eu_5} together with the fact that the mapping $x\mapsto x+g(x,y,z,u)$ is non-decreasing for $x\in[0,1]$, we obtain by a slight modification of Theorem 5.1 in \cite{LP} that there exists a unique strong solution to \eqref{sde_n_2}.
\section{Proof of Proposition \ref{asy_a}}\label{appendix_2}
	We recall that the infinitesimal generator $\mathcal{B}^N$ of the process $Y^N$ is given for any $f\in\mathcal{C}^2([0,1])$ by
	\begin{align*}
		\mathcal{B}^Nf(x)=\int_{\Delta^N}\Bigg\{\frac{\lfloor xN\rfloor}{N}\E\Bigg[f&\left(\frac{\lfloor xN\rfloor}{N}+\frac{1}{N}B_N\right)-f\left(\frac{\lfloor xN\rfloor}{N}\right)\Bigg]\notag\\
		&+\left(1-\frac{\lfloor xN\rfloor}{N}\right)\E\Bigg[f\left(\frac{\lfloor xN\rfloor}{N}-\frac{1}{N}\tilde{B}_N\right)-f\left(\frac{\lfloor xN\rfloor}{N}\right)\Bigg]\Bigg\}\Lambda(dy,dz),
	\end{align*}
	where $B_N:=\text{Binom}\left(N\left(1-\frac{\lfloor xN\rfloor}{N}\right);y\right)$ and $\tilde{B}_N:=\text{Binom}\left(\lfloor xN\rfloor;y+z\right)$.
	
	Now, for any $f\in\mathcal{C}^2([0,1])$ and $x\in[0,1]$ we have by Taylor's Theorem that
	\begin{align}\label{scalin_moran_0_a}
		f\left(\frac{\lfloor xN\rfloor}{N}+\frac{1}{N}B_N\right)&-f\left(\frac{\lfloor xN\rfloor}{N}\right)-(f(x+(1-x)y)-f(x))\notag\\
		&=R\left(\frac{\lfloor xN\rfloor}{N},\frac{\lfloor xN\rfloor}{N}+\frac{1}{N}B_N\right)-R(x,x+(1-x)y).
	\end{align}
	where $R(u,v):=\int_u^vf'(t)dt$.
	
	Now, we note that 
	\begin{align}\label{scalin_moran_1_a}
		R\Bigg(\frac{\lfloor xN\rfloor}{N},\frac{\lfloor xN\rfloor}{N}+\frac{1}{N}B_N\Bigg)-R(x,x+(1-x)y)
		%&=\int_{\frac{\lfloor xN\rfloor}{N}}^{\frac{\lfloor xN\rfloor}{N}+\frac{1}{N}B_N}f'(t)dt-\int_{x}^{x+(1-x)y}f'(t)dt\notag\\
		&\leq\int_{\frac{\lfloor xN\rfloor}{N}}^{x}f'(t)dt+\int_{x+(1-x)y}^{\frac{\lfloor xN\rfloor}{N}+\frac{1}{N}B_N}f'(t)dt.
	\end{align}
	Therefore, using \eqref{scalin_moran_1_a} 
	\begin{align}\label{scalin_moran_2_a}
		\Bigg|R\Bigg(\frac{\lfloor xN\rfloor}{N},\frac{\lfloor xN\rfloor}{N}+\frac{1}{N}B_N\Bigg)&-R(x,x+(1-x)y)\Bigg|\notag\\
		&\leq \|f'\|_{[0,1]}\left|x-\frac{\lfloor xN\rfloor}{N}\right|+\|f'\|_{[0,1]}\left|\frac{\lfloor xN\rfloor}{N}+\frac{1}{N}B_N-x-(1-x)y\right|\notag\\
		%&=\|f'\|_{[0,1]}\frac{1}{N}+\|f'\|_{[0,1]}\left|(1-y)\left(\frac{\lfloor xN\rfloor}{N}-x\right)+\frac{1}{N}B_N-\left(1-\frac{\lfloor xN\rfloor}{N}\right)y\right|\notag\\
		%&+\|f'\|_{\infty}(1-x)y\left|\frac{\lfloor xN\rfloor}{N}+\frac{1}{N}B_N-x-(1-x)y\right|\notag\\
		&\leq\|f'\|_{[0,1]}\frac{1}{N}+\|f'\|_{[0,1]}\left[\frac{1}{N}(1-y)+\left|\frac{1}{N}B_N-\left(1-\frac{\lfloor xN\rfloor}{N}\right)y\right|\right],
	\end{align}
	where $\|g\|_{[0,1]}:=\sup_{x\in[0,1]}|g(x)|$ for any $g\in\mathcal{C}^2([0,1])$. 
	
	On the other hand, by the Cauchy-Schwarz inequality%we note that 
	\begin{align*}
		\E\left[\left|\frac{1}{N}B_N-\left(1-\frac{\lfloor xN\rfloor}{N}\right)y\right|\right] \leq\left(\E\left[\left|\frac{1}{N}B_N-\left(1-\frac{\lfloor xN\rfloor}{N}\right)y\right|^2\right]\right)^{1/2}%=\frac{1}{N^2}\E\left[\left|B_N-N\left(1-\frac{\lfloor xN\rfloor}{N}\right)y\right|^2\right]
		=\left(\frac{1}{N}\left(1-\frac{\lfloor xN\rfloor}{N}\right)y(1-y)\right)^{1/2},
	\end{align*}
	%and by the Cauchy-Schwarz inequality we have
	%\begin{align*}
	%	\E\left[\left|\frac{1}{N}B_N-\left(1-\frac{\lfloor xN\rfloor}{N}\right)y\right|\right]=\left[\frac{1}{N}\left(1-\frac{\lfloor xN\rfloor}{N}\right)y(1-y)\right]^{1/2}.
	%\end{align*}
	Hence, by taking expectations on \eqref{scalin_moran_2_a}, we can find a constant $C_f>0$ only dependent on $f$ such that
	\begin{align}\label{scalin_moran_3_a}
		\E\Bigg[\Bigg|R\Bigg(\frac{\lfloor xN\rfloor}{N},\frac{\lfloor xN\rfloor}{N}+\frac{1}{N}B_N\Bigg)-R(x,x+(1-x)y)\Bigg|\Bigg]=C_f\left(\frac{1}{N}+\frac{1}{N^{1/2}}y^{1/2}\right).
	\end{align}
	Hence, by \eqref{scalin_moran_0_a} together with \eqref{scalin_moran_3_a} there exists a constant $C_f>0$ only dependent on $f$ such that
	\begin{align}\label{freq_ASG_0}
		\Bigg|f\left(\frac{\lfloor xN\rfloor}{N}+\frac{1}{N}B_N\right)&-f\left(\frac{\lfloor xN\rfloor}{N}\right)-(f(x+(1-x)y)-f(x))\Bigg|\leq C_f\left(\frac{1}{N}+\frac{1}{N^{1/2}}y^{1/2}\right).
	\end{align}Very similarly we can deal with the other term in the generator involving $\tilde{B}_N$, and after the same kind of calculations as before we arrive at
	\begin{align}\label{freq_ASG_1}
		\Bigg|f\left(\frac{\lfloor xN\rfloor}{N}-\frac{1}{N}\tilde{B}_N\right)
		-f\left(\frac{\lfloor xN\rfloor}{N}\right)-(f(x-x(y+z))-f(x))\Bigg|\leq \tilde{C}_f\left(\frac{1}{N}+\frac{1}{N^{1/2}}(y+z)^{1/2}\right).
	\end{align}
	Using \eqref{freq_ASG_0} together with \eqref{freq_ASG_1} we obtain, for $x\in[0,1]$,
	\begin{align}\label{freq_ASG_2}
		%\lim_{N\to\infty}
		\sup_{x\in[0,1]}\Bigg|\int_{\Delta^N}&\Bigg\{x\E\Bigg[f\left(\frac{\lfloor xN\rfloor}{N}+\frac{1}{N}B_N\right)-f\left(\frac{\lfloor xN\rfloor}{N}\right)\Bigg]\notag \\
		&\hspace{1.5cm}+(1-x)\E\Bigg[f\left(\frac{\lfloor xN\rfloor}{N}-\frac{1}{N}\tilde{B}_N\right)-f\left(\frac{\lfloor xN\rfloor}{N}\right)\Bigg]\Bigg\}\Lambda(dy,dz)\notag\\
		&-\int_{\Delta}\Bigg\{xf(x+y(1-x))+(1-x)f(x-(y+z)x)-f(x)\Bigg\}\Lambda(dy,dz)\Bigg|\notag\\
		&\leq 2(C_f+\tilde{C}_f)\frac{1}{N^{1/2}}\int_{\Delta^N}\Lambda(dy,dz)\notag\\&+\sup_{x\in[0,1]}\int_{\Delta\backslash\Delta^N}\left|xf(x+y(1-x))+(1-x)f(x-(y+z)x)-f(x)\right|\Lambda(dy,dz).
		%&\leq\lim_{N\to\infty} K\left(\frac{1}{N^{1/2-\alpha}}\int_{\Delta}(y^2+z)\Lambda(dy,dz)-\int_{\Delta\backslash\Delta^N}(y^2+z)\Lambda(dy,dz)\right)=0.
	\end{align}
	Expanding the integrand in the right-hand side of \eqref{freq_ASG_2} gives
	\begin{align}\label{freq_ASG_4}
		|xf(x+y(1-x))&+(1-x)f(x-(y+z)x)-f(x)|=\Bigg|x\int_x^{x+y(1-x)}\frac{f''(t)}{2}(x+y(1-x)-t)dt\notag\\&+(1-x)\int_x^{x-(y+z)x}\frac{f''(t)}{2}(x-(y+z)x-t)dt-x(1-x)zf'(x)\Bigg|\notag\\
		&\leq \frac{\|f''\|_{[0,1]}}{4}\left(x(1-x)^2y^2+(1-x)(y+z)^2x^2\right)+\|f'\|_{[0,1]}x(1-x)z\leq K_f(y^2+z),
	\end{align}
	where $K_f$ is positive constant only dependent on $f$. Hence, using \eqref{freq_ASG_4} in \eqref{freq_ASG_2} 
	\begin{align}\label{freq_ASG_5}
		\lim_{n\to\infty}\sup_{x\in[0,1]}\Bigg|\int_{\Delta^N}\Bigg\{x\E\Bigg[f&\left(\frac{\lfloor xN\rfloor}{N}+\frac{1}{N}B_N\right)-f\left(\frac{\lfloor xN\rfloor}{N}\right)\Bigg]\notag\\&+(1-x)\E\Bigg[f\left(\frac{\lfloor xN\rfloor}{N}-\frac{1}{N}\tilde{B}_N\right)-f\left(\frac{\lfloor xN\rfloor}{N}\right)\Bigg]\Bigg\}\Lambda(dy,dz)\notag\\
		&-\int_{\Delta}\Bigg\{xf(x+y(1-x))+(1-x)f(x-(y+z)x)-f(x)\Bigg\}\Lambda(dy,dz)\Bigg|\notag\\
		&\leq \lim_{n\to\infty}\left[ (C_f+\tilde{C}_f)\frac{1}{N^{1/2-\alpha}}\int_{\Delta}(y^2+z)\Lambda(dy,dz)+K_f\int_{\Delta\backslash\Delta^N}(y^2+z)\Lambda(dy,dz)\right]\notag \\
		&=0.
	\end{align} 
	Finally, we note that
	\begin{align}\label{freq_ASG_3}
		\lim_{N\to\infty}&\sup_{x\in[0,1]}\Bigg|\int_{\Delta^N}\Bigg\{\frac{\lfloor xN\rfloor}{N}\E\Bigg[f\left(\frac{\lfloor xN\rfloor}{N}+\frac{1}{N}B_N\right)-f\left(\frac{\lfloor xN\rfloor}{N}\right)\Bigg]\notag\\
		&+\left(1-\frac{\lfloor xN\rfloor}{N}\right)\E\Bigg[f\left(\frac{\lfloor xN\rfloor}{N}-\frac{1}{N}\tilde{B}_N\right)-f\left(\frac{\lfloor xN\rfloor}{N}\right)\Bigg]\Bigg\}\Lambda(dy,dz)\notag\\
		&-\int_{\Delta^N}\Bigg\{x\E\Bigg[f\left(\frac{\lfloor xN\rfloor}{N}+\frac{1}{N}B_N\right)-f\left(\frac{\lfloor xN\rfloor}{N}\right)\Bigg]\notag \\
		&\hspace{1.5cm}+(1-x)\E\Bigg[f\left(\frac{\lfloor xN\rfloor}{N}-\frac{1}{N}\tilde{B}_N\right)-f\left(\frac{\lfloor xN\rfloor}{N}\right)\Bigg]\Bigg\}\Lambda(dy,dz)\Bigg|\notag\\
		&\leq\lim_{N\to\infty} 2\|f\|_{[0,1]}\frac{1}{N^{1-\alpha}}\int_{\Delta}(y^2+z)\Lambda(dy,dz)=0.
	\end{align}
	Following the steps of Proposition 3 in \cite{CGP} we have that the infinitesimal generator $\mathcal{B}$ of the $\Lambda$-asymmetric frequency process $Y$ is given for any $f\in\mathcal{C}^2([0,1])$ and $x\in[0,1]$ by
	\begin{align}\label{inf_1}
		\mathcal{B}f(x)=\int_{\Delta}\left[xf(x+y(1-x))+(1-x)f(x-(y+z)x)-f(x)\right]\Lambda(dy,dz), \qquad x\in[0,1].
	\end{align}
	Hence, by \eqref{freq_ASG_5} together with \eqref{freq_ASG_3} we obtain that $\mathcal{B}^Nf\to \mathcal{B}f$ as $N\to\infty$ uniformly in $[0,1]$. Therefore, Theorem 17.25 in \cite{Ka} implies that $Y^N\to Y$ as $N\to\infty$ weakly in the space $\mathbb{D}([0,T],[0,1])$.

\end{appendix}

\paragraph*{Acknowledgements} The authors thank two anonymous referees for extremely helpful and constructive comments, that immensely improved the presentation.   AGC was supported by the grant PAPIIT UNAM IN101722 ``Nuevas aplicaciones de la dualidad de momentos y de la construccion Lookdown'', and acknowledge support from the Hausdorff Research Institute for Mathematics in Bonn where he made a 3 month research visit in the summer of 2022. The authors would like to thank Fernanda Lopez, who wrote her Master thesis at UNAM on a simplified version of the model discussed in this manuscript.

	%\section*{References}

\end{document}